\documentclass[11pt]{amsart}
\usepackage[top=1.2in,bottom=1.0in,left=1.0in,right=1.0in]{geometry}
\usepackage{amsmath}
\usepackage{amssymb}
\usepackage{mathrsfs}
\usepackage{times}
\usepackage{bm}
\usepackage{latexsym}
\usepackage{graphicx}
\usepackage{indentfirst}
\usepackage{graphicx}
\usepackage{epsfig}
\usepackage{psfrag}
\usepackage{subfigure}
\usepackage{url}
\usepackage{caption}
\usepackage{float}
\usepackage{stfloats}
\usepackage{fancyhdr}
\usepackage{dsfont}
\usepackage{amsfonts}
\usepackage{enumerate}
\usepackage{epstopdf}
\usepackage{color}
\usepackage{amssymb,amsmath,amsthm,mathrsfs}
\usepackage{booktabs}
\usepackage{threeparttable}
\usepackage[titletoc]{appendix}
\usepackage{algorithm}
\usepackage{algorithmic}
\usepackage{stmaryrd}
\usepackage[english]{babel}
\usepackage{caption}
\usepackage[misc]{ifsym}

 \usepackage{leftidx}
    \numberwithin{equation}{section}%
    \numberwithin{table}{section}%
    \numberwithin{figure}{section}
\allowdisplaybreaks

\newtheorem{lemma}{Lemma}[section]
\newtheorem{remark}{Remark}[section]
\newtheorem{definition}{Definition}[section]

\newtheorem{theorem}{Theorem}[section]

\newcommand{\normmm}[1]{{\left\vert\kern-0.25ex\left\vert\kern-0.25ex\left\vert #1
    \right\vert\kern-0.25ex\right\vert\kern-0.25ex\right\vert}}

\begin{document}

\title{Finite Element Methods For Interface Problems On Local Anisotropic Fitting Mixed Meshes}
\author[J. Hu]{Jun Hu}
\address[Jun Hu]{School of Mathematical Sciences, Peking University, Beijing, China}
\email{hujun@math.pku.edu.cn}

\author[H. Wang]{Hua Wang}
\address[Hua Wang \Letter]{School of Mathematical Sciences, Peking University, Beijing, China}
\email{wanghua.math@foxmail.com}
\thanks{In this research, Jun Hu was supported by NSFC projects 11625101 and 11421101; Hua Wang was supported by China Postdoctoral Science Foundation Grand 2019M660277 and  Jiangsu Key Lab for NSLSCS Grant 201906}

\date{}

\maketitle

\begin{abstract}
A simple and efficient interface-fitted mesh generation algorithm is developed in this paper. This algorithm can produce a local anisotropic fitting mixed mesh which consists of both triangles and quadrilaterals near the interface. A new finite element method is proposed for second order elliptic interface problems based on the resulting mesh. Optimal approximation capabilities on anisotropic elements are proved in both the $H^1$ and $L^2$ norms. The discrete system is usually ill-conditioned due to anisotropic and small elements near the interface. Thereupon, a multigrid method is presented to handle this issue. The convergence rate of the multigrid method is shown to be optimal with respect to both the coefficient jump ratio and mesh size. Numerical experiments are presented to demonstrate the theoretical results.
\end{abstract}

\section{Introduction}
Let $\Omega$ be a convex polygon in $\mathbb{R}^{2}$ which is separated by a $C^{2}$-continuous interface $\Gamma$ into two sub-domains $\Omega_{1}$ and $\Omega_{2}$, see Figure \ref{fig.domain} for an illustration. Consider the following interface problem
\begin{eqnarray}\label{elliptic}
\begin{aligned}
 -\mathrm{div}(\beta\nabla u)&=f ~& \mathrm{in}~&\Omega_{1}\cup\Omega_{2},\\
 [\![u]\!]&=q ~&\mathrm{on}~&\Gamma,\\
 [\![\beta\frac{\partial u}{\partial \bm{n}_{\Gamma}}]\!]&=g~&\mathrm{on}~&\Gamma,\\
 u&=0~&\mathrm{on}~&\partial  \Omega,
\end{aligned}
\end{eqnarray}
where $[\![v]\!]:=(v|_{\Omega_{1}})|_{\Gamma}-(v|_{\Omega_{2}})|_{\Gamma}$ for any $v$ belonging to $H^{1}(\Omega_1\cup\Omega_2)$, and $\bm{n}_{\Gamma}$ is the unit normal vector of $\Gamma$ which points from $\Omega_{1}$ to $\Omega_{2}$, see Figure \ref{fig.domain}.
The coefficient function $\beta$ is discontinuous across the interface $\Gamma$, i.e.,
\begin{eqnarray}\label{beta}
\beta=\left\{
\begin{aligned}
&\beta_{1},~~\text{in}~\Omega_{1},\\
&\beta_{2},~~\text{in}~ \Omega_{2},
\end{aligned}
\right.
\end{eqnarray}
where $\beta_1$ and $\beta_2$ are positive constants.
\begin{figure}[H]
\centering
  \includegraphics[width=0.3\textwidth]{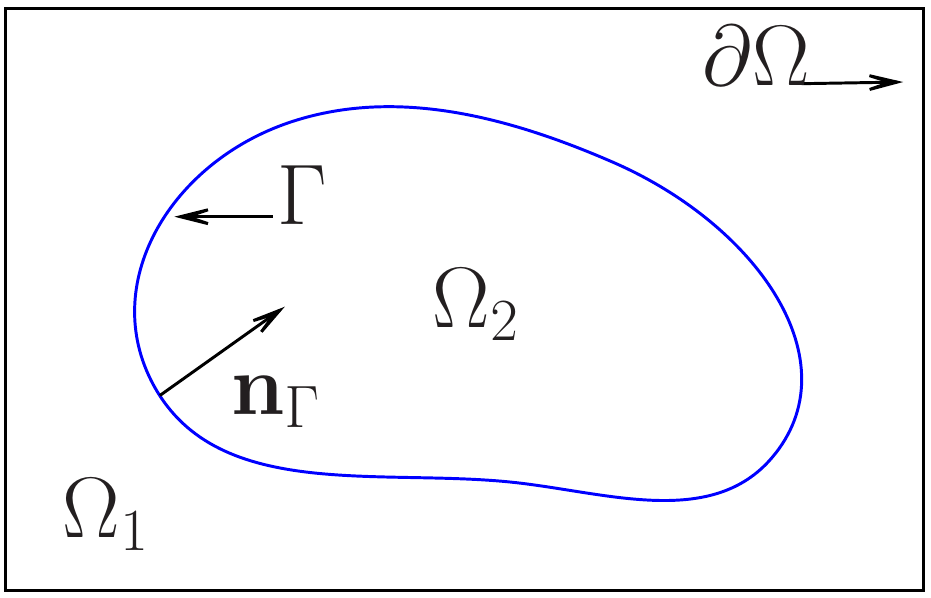}
  \caption{A sketch of the domain for the interface problem.}\label{fig.domain}
\end{figure}
This problem occurs widely in practical applications, such as fluid mechanics, electromagnetic wave propagations, materials sciences, and biological sciences. Mathematically, the interface problem usually leads to partial differential equations with discontinuous or non-smooth solutions across interfaces. Hence, classical numerical methods designed for smooth solutions do not work efficiently. For regularity of the interface problem \eqref{elliptic}, Chen and Zou \cite{Chen1998} proved that
\begin{equation*}
\|u\|_{H^2(\Omega_1)}+\|u\|_{H^2(\Omega_2)}\leq C_\beta (\|f\|_{L^2(\Omega)}+\|g\|_{H^{1/2}(\Gamma)}),
\end{equation*}
where $C_\beta$ is a constant independent of $u$, $f$ and $g$, but depends strongly and implicitly on the constants $\beta_1$, $\beta_2$ and the jump in the coefficient across the interface. Later, Huang and Zou \cite{Huang2002,Huang2007} improved the estimate
\begin{equation*}
|\beta_1^{1/2}u|_{H^2(\Omega_1)}+|\beta_2^{1/2}u|_{H^2(\Omega_2)}\leq C (\|f\|_{L^2(\Omega)}+\|g\|_{H^{1/2}(\Gamma)}).
\end{equation*}

For more than four decades, there has been a growing interest in the interface problem, and a vast of literature is available, see \cite{Babuska1970,Xu1982,Chen1998,Li1998,Belytschko1999,Hansbo2002,Li2003,Zi2003,Hansbo2014,
Adjerid2015,Kergrene2016,Guzman2016,Burman2016}. There are two major classes of finite element methods (FEM) for the interface problem, namely, the interface-fitted FEMs and the interface-unfitted FEMs, categorized according to the topological relation between discrete grids and the interface. Accuracy would be lost if standard finite element methods are used on interface-unfitted meshes. One way to recover the approximate accuracy is to use interface-fitted meshes, such as \cite{Xu1982,Chen1998}. Another way is to modify finite element spaces near the interface, see \cite{Li1998,Zi2003,Hansbo2004}. In \cite{Xu2016}, the authors proposed linear element schemes for diffusion equations and the Stokes equation with discontinuous coefficients on an interface-fitted grid which satisfies the maximal angle condition. Interface-fitted mesh generation algorithms which can produce a semi-structured interface-fitted mesh in two and three dimensions are proposed by Chen et al. In \cite{Chen2017mesh} for interface problems, virtual element methods are applied to solve the elliptic interface problem. With the assumption that each subdomain is an open disjointed polygonal or polyhedral region, Xu and Zhu \cite{Xu2008} proved a uniform convergence rate with respect to the mesh size and the coefficient jump ratios.

This work focuses on the interface-fitted mesh approach. Let $\mathcal{U}_h$ be a general quasi-uniform interface-unfitted mesh of $\Omega$, an interface-fitted mesh $\mathcal{F}_h$ can be generated by simply connecting the intersected points of $\Gamma$ and the mesh $\mathcal{U}_h$ successively. By this way, an interface-fitted mesh can be generated quite efficiently no matter how complicated the interface is. It is obvious that the interface-fitted mesh $\mathcal{F}_h$ contains some anisotropic triangles and quadrilaterals near the interface. By using the approximation results for anisotropic triangle elements proposed by Babuska \cite{Babuska1976} and anisotropic quadrilaterals elements proposed by Acosta and Duran \cite{Acosta2000}, optimal approximation capability of finite element spaces have been proved. Since these anisotropic elements and coefficient jump ratio may ¡¯stiffen¡¯ the matrix, a multigrid method is presented for solving the discrete system. Without assuming that each subdomain is an open disjointed polygonal as Xu and Zhu \cite{Xu2008} do, the convergence rate of the multigrid method is shown to be optimal with respect to the jump ratio and mesh size even though the interface $\Gamma$ is a general $C^2$-continuous curve.

The rest of this paper is organized as follows. In Section \ref{notation}, we introduce some notations which will be frequently used in this paper. In Section \ref{fem}, we present the local anisotropic finite element space and the weak form for the non-homogeneous second order elliptic interface problem. In Section \ref{erranaly}, we derive an a prior estimate for the finite element method based on local anisotropic fitting mixed meshes. In Section \ref{multigrid}, we propose a multigrid solver for the resulting linear system. We also present some numerical examples in Section \ref{sec:example} to validate our theoretical results.

%Since it may be computationally inconvenient to fit the mesh to $\Gamma$, we wish to consider a finite-element approximation based on a mesh that is independent of $\Gamma$. analysed the Galerkin approximation of (1.4) in the absence of variational crimes; that is, it was assumed that integrals over $\Omega_i$, and $\Gamma$ could be performed exactly. this analysis was for the unpractical standard Galerkin method in the absence of variational crimes

\section{Notation and Definitions}\label{notation}
%Let $\mathcal{D}(\Omega)$ be the linear space of infinitely differentiable functions, with compact support on $\Omega$, and $\mathcal{D}'(\Omega)$ denote the dual space of $\mathcal{D}(\Omega)$. For $r\geq 0$ and $1\leq p\leq \infty$.
For integer $r\geq 0$, define the piecewise $H^{r}$ Sobolev space
\begin{equation*}
H^{r}(\Omega_{1}\cup\Omega_{2})=\{v\in L^2(\Omega);v|_{\Omega_{i}}\in H^{r}(\Omega_{i}),i=1,2\},
\end{equation*}
equipped with the norm and semi-norm
\begin{eqnarray*}\label{norm}
\begin{aligned}
\|v\|_{H^{r}(\Omega_{1}\cup\Omega_{2})}&=(\|v\|_{H^{r}(\Omega_1)}^{p}
+\|v\|_{H^{r}(\Omega_2)}^{p})^{1/p},\\
|v|_{H^{r}(\Omega_{1}\cup\Omega_{2})}&=(|v|_{H^{r}(\Omega_1)}^{p}
+|v|_{H^{r}(\Omega_2)}^{p})^{1/p}.
\end{aligned}
\end{eqnarray*}
Furthermore, let $\tilde{H}^{r}(\Omega_1\cup\Omega_2)=H^{1}_{0}(\Omega)\cap H^{r}(\Omega_{1}\cup\Omega_{2}).$
%In the case $p=2$, we use $\tilde{H}^{r}(\Omega_{1}\cup\Omega_{2})$ to represent $\tilde{W}^{r,2}(\Omega_{1}\cup\Omega_{2})$.

In the following, a simple and efficient method is presented for generating interface-fitted mesh. Figures \ref{mesh0}-\ref{mesh1} show how to obtain an interface-fitted grid. The domain is a square and the interface is a circle. $\mathcal{U}_h$ is an quasi uniform mesh which does not fit to the interface. By connecting intersected points of $\Gamma$ and $\mathcal{U}_{h}$ successively, a resolution $\Gamma_h$ (the red line) of $\Gamma$ (the blue line) is obtained. The generated mesh $\mathcal{F}_h$ is an interface-fitted mesh which contains anisotropic triangles and quadrilaterals near the interface.
%\begin{algorithm}
%\caption{Generate an interface-fitted mesh $\mathcal{F}_h$.}
%\begin{algorithmic}[]\label{fmesh}
%\STATE (1) Generate an interface-unfitted quasi-uniform mesh $\mathcal{U}_h$ with its mesh size equal $h$;
%\STATE (2) Connect intersected points of $\Gamma$ and $\mathcal{U}_{h}$ successively to generate the interface-fitted grid $\mathcal{F}_h$;\\
%\end{algorithmic}
%\end{algorithm}

\begin{figure}[H]
\begin{minipage}[t]{0.5\linewidth}
\centering
\includegraphics[width=1\textwidth]{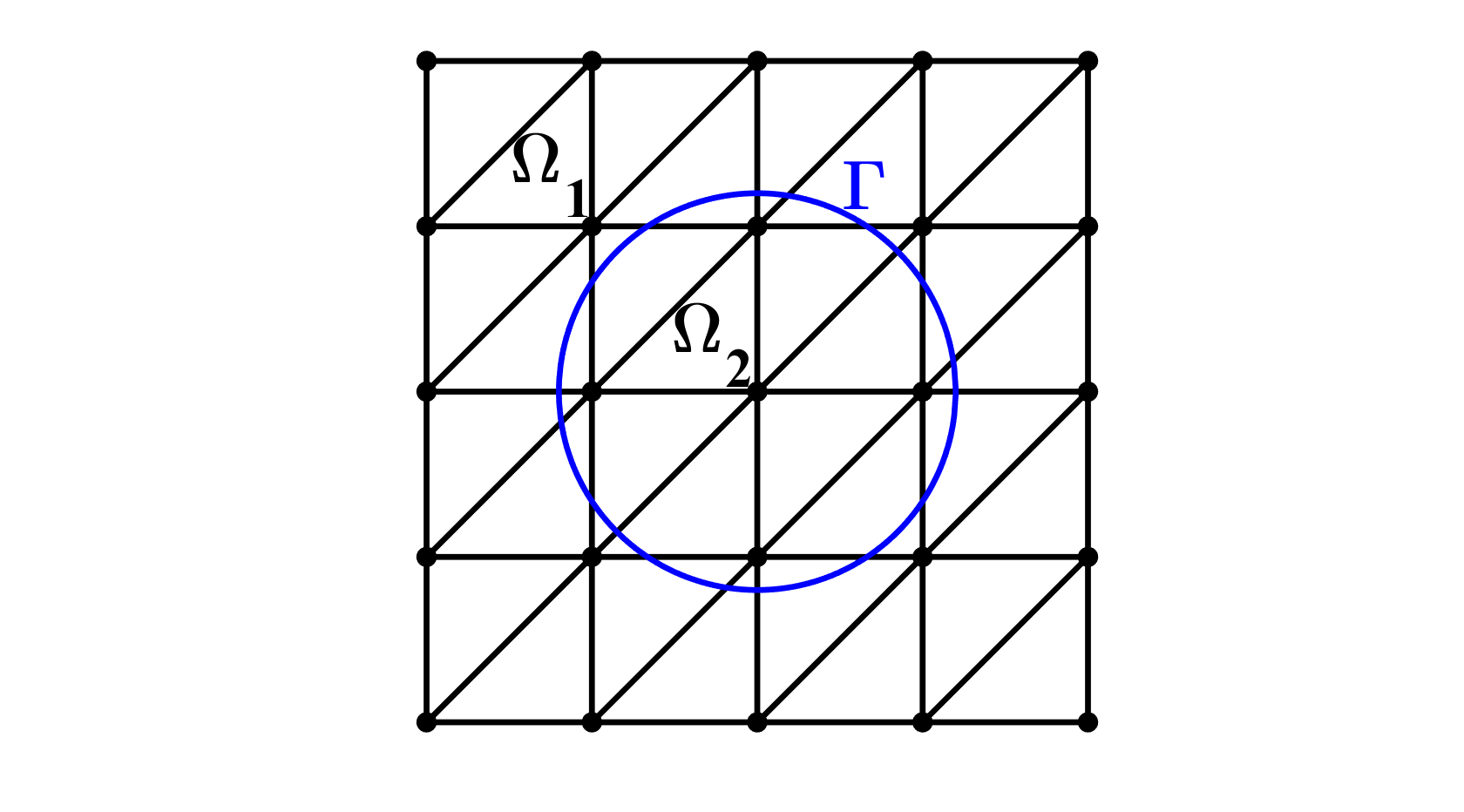}
\caption{An interface-unfitted mesh $\mathcal{U}_h$.}\label{mesh0}
\label{mesh0}
\end{minipage}%
\begin{minipage}[t]{0.5\linewidth}
\centering
\includegraphics[width=1\textwidth]{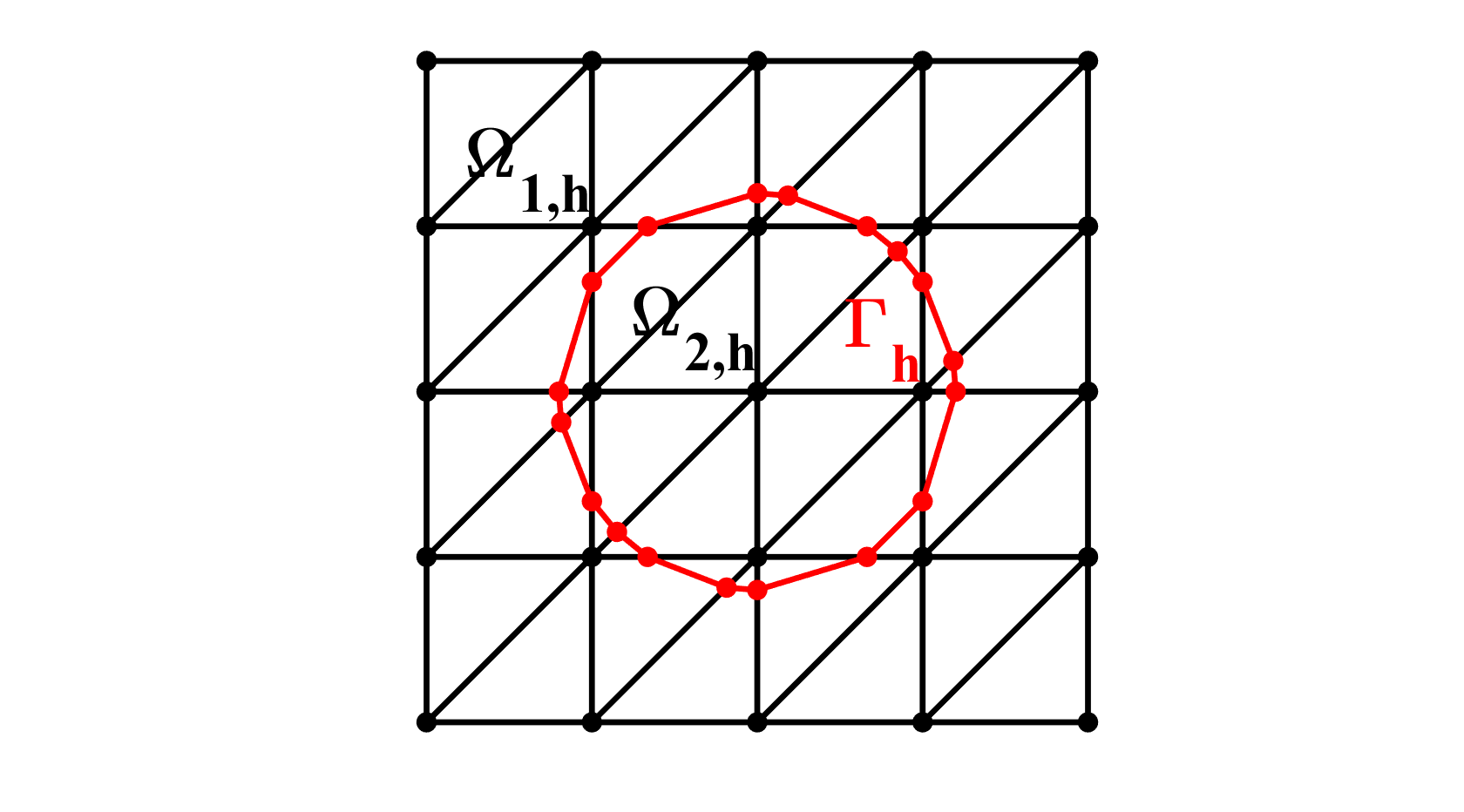}
\caption{An interface-fitted mesh $\mathcal{F}_h$.}\label{mesh1}
\label{mesh1}
\end{minipage}
\end{figure}
Let $\Omega_{2,h}$ be an approximation of $\Omega_2$, and $\Omega_{1,h}$ stand for the domain with $\partial \Omega$ and $\Gamma_h$ as its exterior and interior boundaries, respectively (see Figure \ref{mesh1}). Then, the domain $\Omega$ is separated into two sub-domains $\Omega_{1,h}$ and $\Omega_{2,h}$. The collection of interface elements in $\mathcal{U}_h$ and $\mathcal{F}_h$ is defined as
\begin{align*}
\mathcal{U}_h^I=\{T\in \mathcal{U}_h; meas_1(T\cap\Gamma)> 0\},\\
\mathcal{F}_h^I=\{T\in \mathcal{F}_h; meas_1(T\cap\Gamma)> 0\},
\end{align*}
where $meas_d$ denotes the $d$-dimensional measure. For any interface element $K\subset \Omega_{i,h}$, let $K_i = K\cap \Omega_i$ and $K_* = K\setminus K_i$, see Figures \ref{ielement1}-\ref{ielement2}.
\begin{figure}[H]
\begin{minipage}[t]{0.5\linewidth}
\centering
\includegraphics[width=0.6\textwidth]{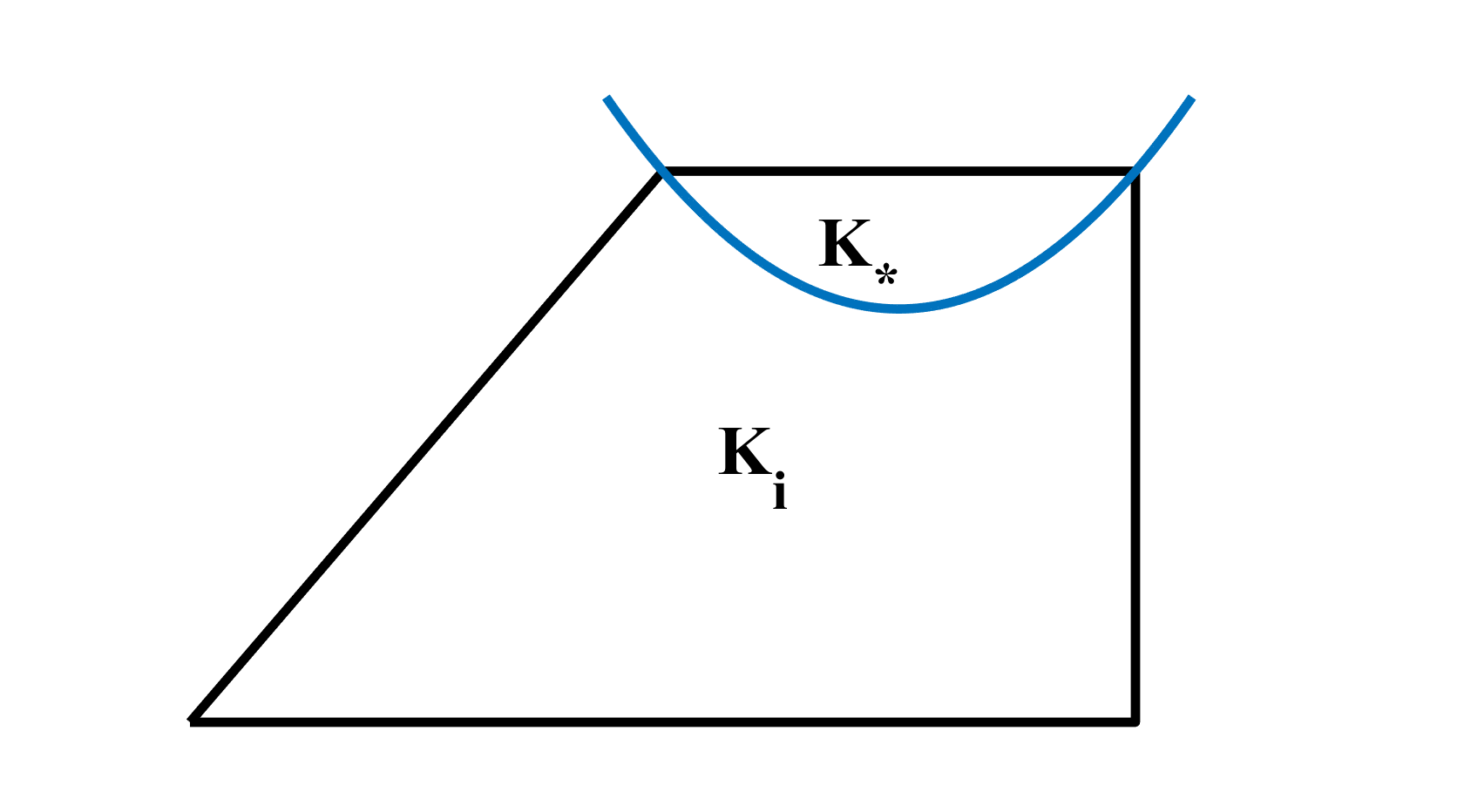}
\caption{A quadrilateral interface element in $\mathcal{F}_h$.}\label{ielement1}
\end{minipage}%
\begin{minipage}[t]{0.5\linewidth}
\centering
\includegraphics[width=0.6\textwidth]{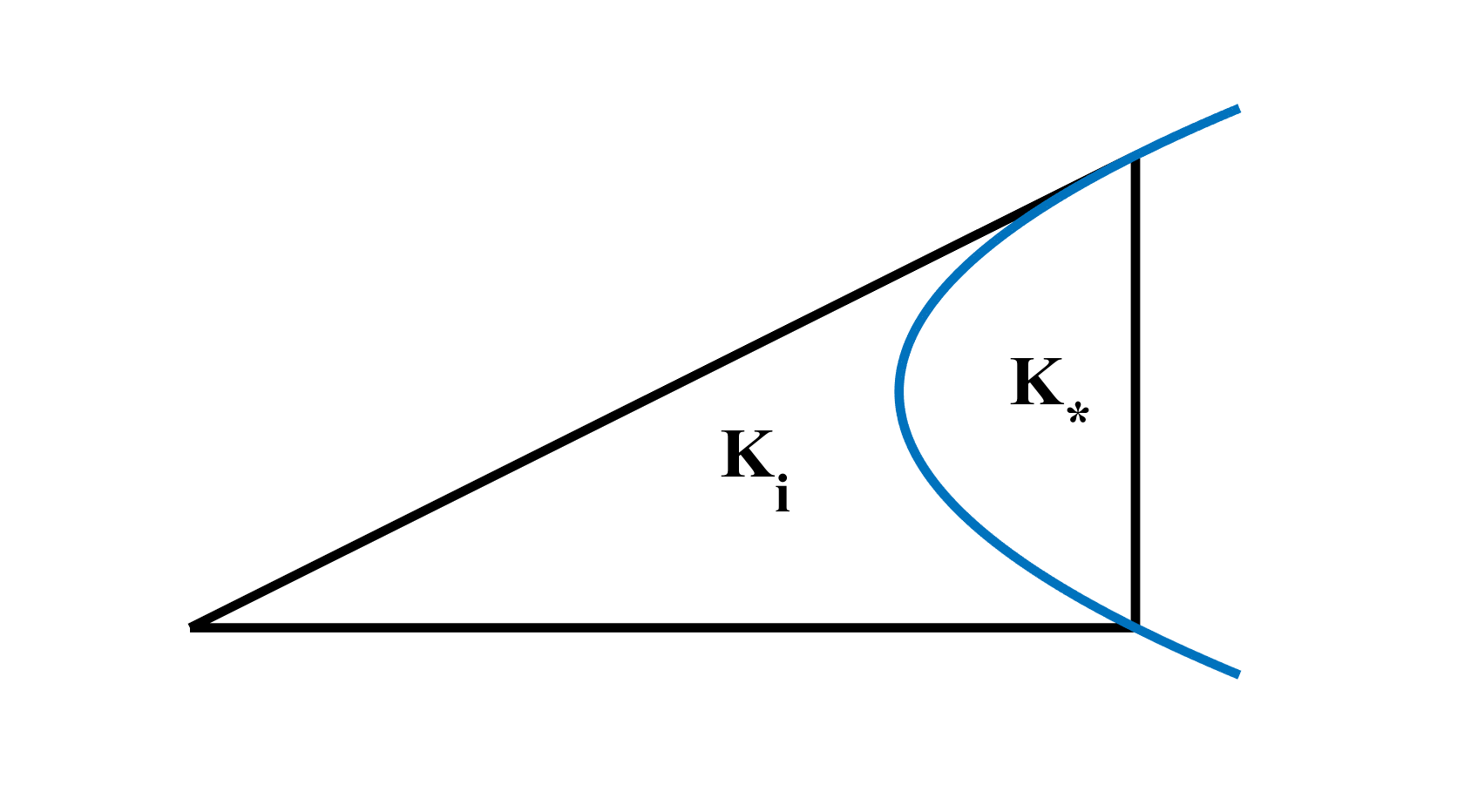}
\caption{A triangle interface element in $\mathcal{F}_h$.}\label{ielement2}
\end{minipage}
\end{figure}
Let $\Omega_{h}^I$ be the region enclosed by $\Gamma$ and $\Gamma_h$, i.e., $\Omega_h^I = \bigcup_{K\in \mathcal{F}_h^I} K_*$. For any function $v\in H^{2}(\Omega_1\cup\Omega_2)$, let $v_i=v|_{\Omega_i}$, $i=1,2$. By the extension theorem \cite{Adams2009}, there exists an operator $E:H^2(\Omega_i)\rightarrow H^2(\Omega)$ such that
\begin{equation*}
Ev_i|_{\Omega_{i}}=v_i,\quad\|Ev_i\|_{H^2(\Omega)}\lesssim \|v_i\|_{H^2(\Omega_i)},
\end{equation*}
for $i=1,2$ (see \cite{Adams2009} for details).  Here the notation $A\lesssim B$ represents the statement $A\leq$ constant $\times B$, where the constant is always independent of the mesh sizes of the triangulations and the location of the interface intersected with the mesh.

In order to analyze the interpolation error for these anisotropic triangles and quadrilaterals, the following are some basic definitions which mainly follow \cite{Babuska1976} and \cite{Acosta2000}.
\begin{definition}[Minimum angle condition]
We say that a quadrilateral $K$ (resp., a triangle $T$) satisfies $Minac(\alpha)$, if the angles of $K$ (resp., $T$) are greater than or equal to $\alpha$. Similarly, we say that a mesh $\mathcal{F}_h$ satisfies $Minac(\alpha)$, if there exists a uniform $\alpha\in(0,\pi]$ such that any $T\in \mathcal{F}_h$ satisfies $Minac(\alpha)$.
\end{definition}

\begin{definition}[Maximum angle condition]
We say that a quadrilateral $K$ (resp., a triangle $T$) satisfies $Maxac(\psi)$, if the angles of $K$ (resp., $T$) are less than or equal to $\psi$. Similarly, we say that a mesh $\mathcal{F}_h$ satisfies $Maxac(\alpha)$, if there exists a uniform $\alpha\in(0,\pi]$ such that any $T\in \mathcal{F}_h$ satisfies $Maxac(\alpha)$.
\end{definition}

\begin{definition}
Let $K$ be a convex quadrilateral. We say that K satisfies the regular decomposition property with constants $N\in R$ and $0<\psi<\pi$, or shortly $RDP(N,\psi)$, if we can divide $K$ into two triangles along one of its diagonals, which will always be called $d_1$, in such a way that $|d_2|/|d_1|\leq N$ and both triangles satisfy $Maxac(\psi)$.
\end{definition}
%For a convex quadrilateral $ABCD$ (see Figure \ref{noRDP}), if $|AC|\ll|BD|$, then it does not satisfies $RDP(N,\psi)$ with a bounded $N$.
%\begin{figure}[H]
%\centering
%\includegraphics[width=3in]{noRDP.eps}
%\caption{A quadrilateral does not satisfy RDP}
%\label{noRDP}
%\end{figure}

\section{Finite element methods for the elliptic interface problem}\label{fem}
\subsection{Finite element space}
For a general convex quadrilateral $K$, denote its vertices by $M_i$ in anticlockwise order. Let $\hat{K}$ be the reference unit square, $F_K:\hat{K}\rightarrow K$ be the transformation defined by
\begin{equation}
F_K(\hat{\mathbf{x}})=\sum\limits_{i=1}^4 M_i\hat{\phi_i}(\hat{\mathbf{x}}),
\end{equation}
where $\hat{\phi}_1=(1-\hat{x})(1-\hat{y}), \hat{\phi}_2=\hat{x}(1-\hat{y}), \hat{\phi}_3=\hat{x}\hat{y}, \hat{\phi}_4=(1-\hat{x})\hat{y}$.
Observe that, $F_K$ is a bijection from the unit square $\hat{K}$ onto the quadrilateral $K$.
\begin{figure}[H]
\vspace{-1cm}
\centering
  \includegraphics[width=0.75\textwidth]{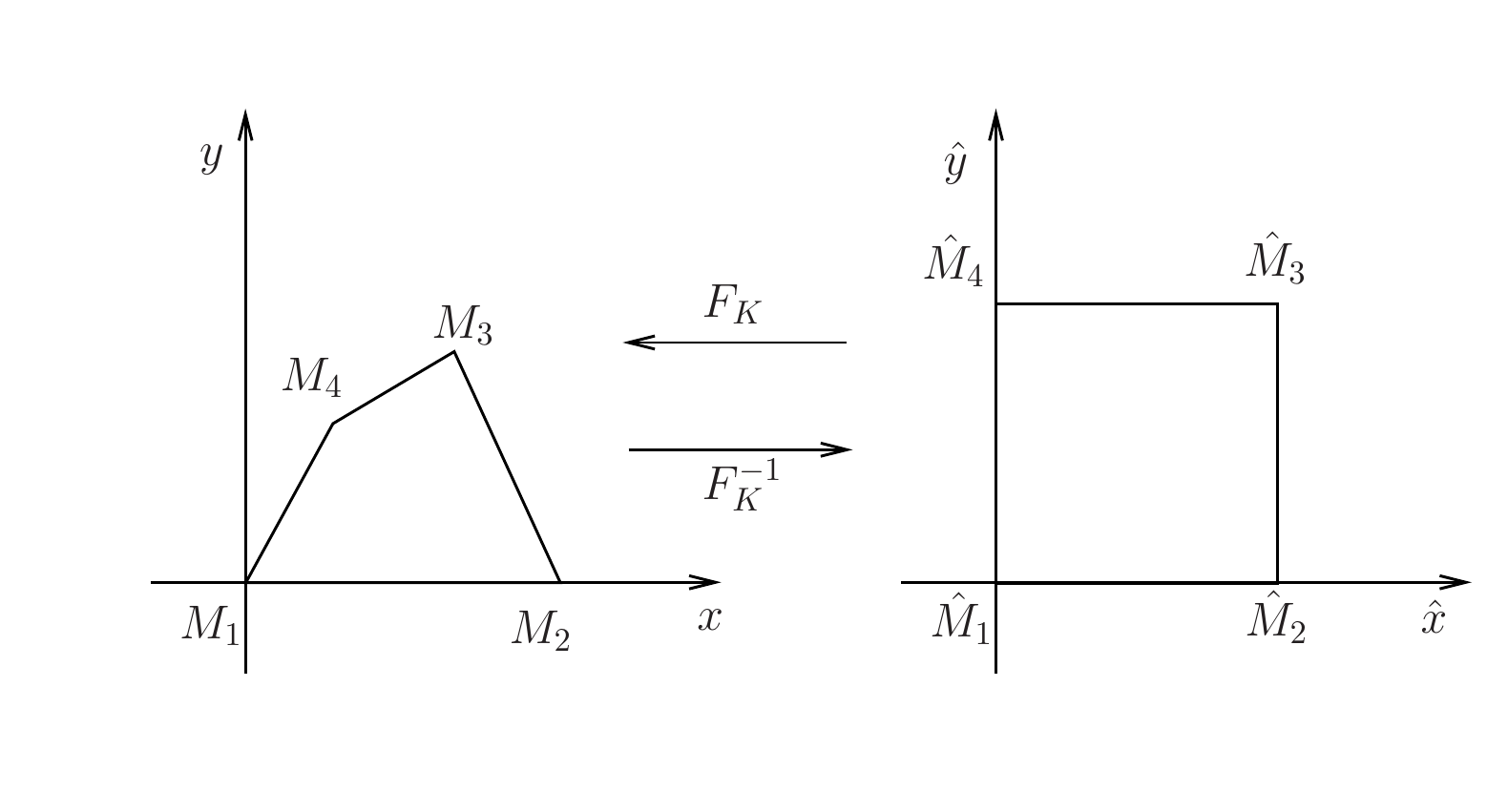}
  \vspace{-1cm}
  \caption{A transformation from $\hat{K}$ to $K$.}
\end{figure}
The basis functions on $K$, no longer bilinears in general, are defined by $\phi_i (x)=\hat{\phi}_i(F_K^{-1}(x))$. Thus, the shape function space on $K$ is defined by
\begin{equation*}
\mathcal{Q}_1(K)=span\{\phi_i,1\leq i\leq 4\}.
\end{equation*}

Similarly, denote the linear shape function space on a general triangle $T$ by $\mathcal{P}_1(T)$, i.e.,
\begin{equation*}
\mathcal{P}_1(T)=span\{1,x,y\}.
\end{equation*}
Thus the finite element space defined on $\mathcal{F}_h$ can be written as
\begin{equation}\label{aspace}
V_h=\{v_h\in C_0(\Omega); v_h|_T\in \mathcal{P}_1(T), v_h|_K\in \mathcal{Q}_1(K)~\forall T,K\in \mathcal{F}_h\}.
\end{equation}
\subsection{Weak form}
By the extension theorem, for any $q\in H^{3/2}(\Gamma)$, there exists a function $z_0\in H^2(\Omega_2)$ s.t.
\begin{equation*}
z_0|_{\Gamma} = -q,\quad \|z_0\|_{2,\Omega_2}\lesssim \|q\|_{3/2,\Gamma}.
\end{equation*}
Let
\begin{eqnarray*}
z=\left\{
\begin{aligned}
&0,~~~~&\text{in}~~~\Omega_{1},\\
&z_0,~~~~&\text{in}~~~ \Omega_{2},
\end{aligned}
\right.
\end{eqnarray*}
the non-homogeneous problem \eqref{elliptic} can be rewritten as
\begin{eqnarray}
\begin{aligned}\label{elliptic2}
 -\mathrm{div}(\beta\nabla \tilde{u})&=f+\mathrm{div}(\beta\nabla z) ~&\mathrm{in}~&\Omega_{1}\cup\Omega_{2},\\
 [\![\tilde{u}]\!]&=0~& \mathrm{on}~&\Gamma,\\
 [\![\beta\frac{\partial \tilde{u}}{\partial n}]\!]&=g-[\![\beta\frac{\partial z}{\partial n}]\!]~ & \mathrm{on}~&\Gamma,\\
 \tilde{u}&=0~& \mathrm{on}~&\partial  \Omega,
\end{aligned}
\end{eqnarray}
with $\tilde{u}=u-z$. The variational formulation for the homogeneous problem \eqref{elliptic2} is: find $\tilde{u}\in V:=H^1_0(\Omega)$ such that
\begin{equation}\label{cweak}
a(\tilde{u},v)=F(v)\quad\forall v\in V,
\end{equation}
where
\begin{align*}
&a(\tilde{u},v)=\int_{\Omega_1\cup\Omega_2}\beta\nabla \tilde{u}\cdot\nabla vdx,\\
&F(v)=\int_{\Omega}fvdx+\int_{\Gamma}gvds-\int_{\Omega_2}\beta_2\nabla z\cdot\nabla vdx.
\end{align*}
Let $\{O_i\}_{i=1}^{m}$ be the set of all nodes of the triangulation $\mathcal{F}_h$ lying on the interface $\Gamma$ (the red points in Figure \ref{mesh1}), and $\{\phi_i\}_{i=1}^m$ the set of nodal basis functions
associated to $\{O_i\}_{i=1}^{m}$. Let $\pi_h: C_0(\Omega)\mapsto V_h$ be the nodal interpolation operator, i.e.,
\begin{equation}
\pi_h u\in V_h,~~ \pi_h u(M)=u(M) ~\forall M\in \mathcal{O},
\end{equation}
where $\mathcal{O}$ is the set of nodal points of $\mathcal{F}_h$. Assume $g$ and $q$ belong to $C(\Gamma)$, define
\begin{align}
&g_h = \sum_{i=1}^m g(O_i)\phi_i,\\
&z_h|_{\Omega_{2,h}} = \pi_h z, \quad z_h|_{\Omega_{1,h}}=0.
\end{align}
Usually, the integrals over $\Omega_i$, and $\Gamma$ could be performed exactly. A reasonable discrete weak form for the interface problem \eqref{elliptic2} with variational crimes is: find $\tilde{u}_h\in V_h$ such that
\begin{equation}\label{weak1}
a_h(\tilde{u}_h,v_h)=\tilde{F}_h(v_h)\quad\forall v_h\in V_h,
\end{equation}
where
\begin{align*}
a_h(\tilde{u}_h,v_h)&=\int_{\Omega_{1,h}}\beta_1\nabla \tilde{u}_h \cdot\nabla v_hdx + \int_{\Omega_{2,h}}\beta_2\nabla \cdot\tilde{u}_h \nabla v_hdx\\
\tilde{F}_h(v_h)&=\int_{\Omega}fv_hdx+\int_{\Gamma_h}g_hv_hds-\int_{\Omega_{2,h}}
\beta_2\nabla z_h\cdot\nabla v_hdx.
\end{align*}
Let $u_h=\tilde{u}_h+z_h$, it is obvious that $u_h$ is an suitable approximation of $u$. However, $z$ is a unknown function, therefore $z_h$ is unknown either. Divide $z_h$ into two parts
\begin{equation*}
z_h= z_{h,\Gamma}+z_{h,0},
\end{equation*}
where
\begin{equation*}
z_{h,\Gamma}|_{\Omega_{2,h}} = -\sum_{i=1}^m q(O_i)\phi_i,
\quad z_{h,\Gamma}|_{\Omega_{1,h}}=0.
\end{equation*}
Therefore, the discrete weak formulation \eqref{weak1} can be rewritten as: find $\bar{u}_h\in V_h$ such that
\begin{equation}\label{weak2}
a_h(\bar{u}_h,v_h)=\bar{F}_h(v_h)\quad\forall v_h\in V_h,
\end{equation}
where
\begin{equation*}
\bar{F}_h(v_h)=\int_{\Omega}fv_hdx+\int_{\Gamma_h}gv_hds-\int_{\Omega_{2,h}}\beta_2\nabla z_{h,\Gamma}\cdot\nabla v_hdx.
\end{equation*}

\begin{lemma}\label{uh}
Let $\tilde{u}_h$ and $\bar{u}_h$ be the solutions of Problem \eqref{weak1} and \eqref{weak2}, respectively. The following relation holds
\begin{equation}\label{weakrelation}
u_h = \tilde{u}_h+z_h=\bar{u}_h+z_{h,\Gamma}
\end{equation}
\end{lemma}
\begin{proof}
Subtracting \eqref{weak2} from \eqref{weak1} yields that
\begin{equation*}
a_h((\tilde{u}_h+z_{h,0})-\bar{u}_h,v_h)=0,~~\forall v_h\in V_h.
\end{equation*}
Since $z_{h,0}\in V_h$, it is straightforward to see that
\begin{equation*}
\tilde{u}_h+z_{h,0}=\bar{u}_h,
\end{equation*}
and consequently \eqref{weakrelation} holds.
\end{proof}
%Therefore, we can rewrite the definition of $u_h$ as
%\begin{equation*}\label{uhdef2}
%u_h= \bar{u}_h+z_{h,\Gamma}.
%\end{equation*}
\section{Error Analysis}\label{erranaly}
\subsection{Interpolation error estimates}
Interpolation estimates are fundamental in finite element error analysis. Since the interpolation estimates on non-interface elements are standard, one only needs to consider the anisotropic elements near the interface.
\begin{lemma}[\cite{Babuska1976}, Theorem 2.3]\label{trierror}
Let $T$ be a triangle with diameter $h$, if $v\in H^2(T)$, there exists a constant $C$ independent of $T$ such that
\begin{equation}
\|v-\pi_h v\|_{L^2(T)}\leq C h^2|v|_{H^2(T)},
\end{equation}
and, if $T$ satisfies $MAC(\psi)$, then there exists a constant $C(\psi)$ which only depends on $\psi$ such that
\begin{equation}
|v-\pi_h v|_{H^1(T)}\leq C(\psi)h|v|_{H^2(T)}.
\end{equation}
\end{lemma}
\begin{lemma}[\cite{Acosta2000}, Theorem 4.7]\label{quaderror}
Let $K$ be a convex quadrilateral with diameter $h$, if $v\in H^2(T)$, there exists a constant $C$ independent of $K$ such that
\begin{equation}
\|v-\pi_h v\|_{L^2(K)}\leq C h^2|v|_{H^2(K)},
\end{equation}
and, if $K$ satisfies $RDP(N,\psi)$, then there exists a constant $C(N,\psi)$ which depends on $N$ and $\psi$ such that
\begin{equation}
|v-\pi_h v|_{H^1(K)}\leq C(N,\psi)h|v|_{H^2(K)}.
\end{equation}
\end{lemma}
\begin{figure}[H]
\centering
  \includegraphics[width=0.45\textwidth]{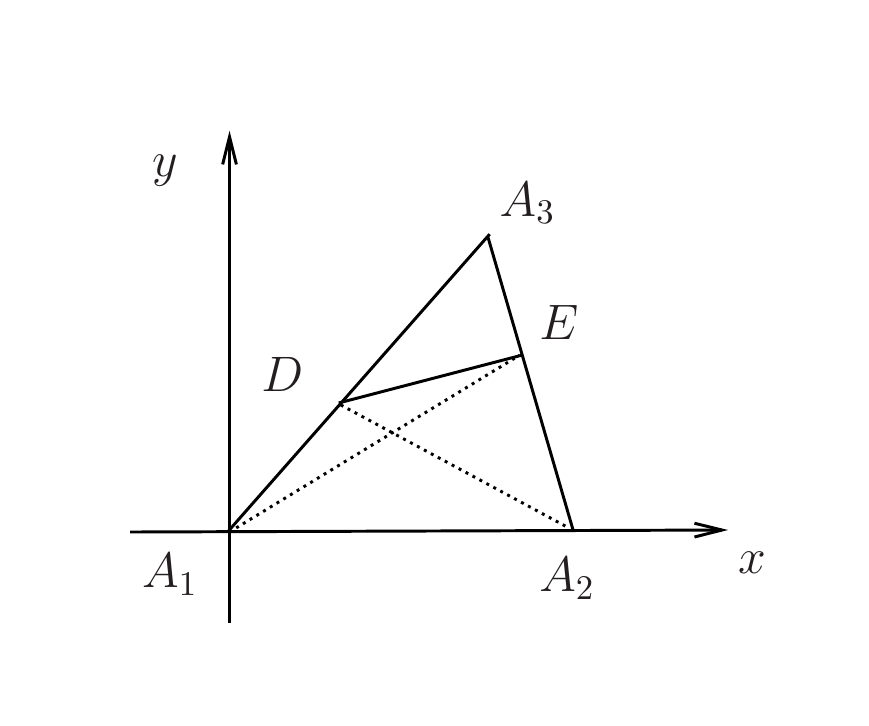}
  \caption{An interface element in $\mathcal{U}_h^I$.}\label{fig.element}
\end{figure}
\begin{lemma}\label{rdp}
If $\mathcal{U}_h$ satisfies $Minac(\alpha)$, then for an arbitrary interface element $T:=\triangle A_1A_2A_3\in \mathcal{U}_h^I$ (see Figure \ref{fig.element}), quadrilateral $A_1 A_2 ED$ satisfies $RDP(N_\alpha,\psi_\alpha)$, where $N_\alpha$ and $\psi_\alpha$ depend only on $\alpha$.
\end{lemma}
\begin{proof}
Without loss of generality, assume $A_1$ is the origin and line $A_1 A_2$ lies on the $x$-axis. Let the coordinates at $A_1, A_2, A_3, D, E$ be
\begin{equation*}
(0, 0), (x_2, 0), (x_3, y_3), (x_D, y_D), (x_E, y_E),
\end{equation*}
respectively, and suppose $\frac{|A_3 D|}{|A_3 A_1|}\geq\frac{|A_3 E|}{|A_3 A_2|}$, see Figure \ref{fig.element} for an illustration.
Divide quadrilateral $A_1 A_2 DE$ into two triangles, $\triangle A_1A_2D$ and $\triangle A_2ED$. Since $\mathcal{U}_h$ satisfies $Minac(\alpha)$, it follows that $\alpha\leq\angle DA_1A_2\leq\pi-\alpha$. Hence triangle $\triangle A_1A_2D$ satisfies $Maxac(\pi-\alpha)$. Since $\frac{|A_3 D|}{|A_3 A_1|}\geq\frac{|A_3 E|}{|A_3 A_2|}$, it follows that $\angle A_3ED\geq \angle A_3A_2A_1$. And consequently
\begin{equation*}
\alpha\leq\angle A_1A_3A_2< \angle DEA_2=\pi-\angle A_3ED\leq \pi-\angle A_3A_2A_1\leq\pi-\alpha,
\end{equation*}
thus triangle $\triangle A_2ED$ satisfies $Maxac(\pi-\alpha)$.
Moreover,
\begin{eqnarray*}
h \sin\alpha \lesssim |A_1A_3||\sin\angle A_1A_3A_2|\leq |A_1E|\leq \max\{|A_1A_2|,|A_1A_3|\}\lesssim h,\\
h \sin\alpha \lesssim |A_2A_3||\sin\angle A_1A_3A_2|\leq |A_2D|\leq \max\{|A_1A_2|,|A_2A_3|\}\lesssim h.
\end{eqnarray*}
Therefore, it is easy to derive that
\begin{equation*}
\frac{|A_1E|}{|A_2D|}\leq \frac{C}{\sin\alpha}.
\end{equation*}
Let $\psi_\alpha=\pi-\alpha$ and $N_\alpha=\frac{C}{\sin\alpha}$, it completes the proof.
\end{proof}

\begin{lemma}
Assume $\mathcal{U}_h$ is a quasi-uniform mesh which satisfies $Minac(\alpha)$, and $\mathcal{F}_h$ is the interface-fitted mesh generated from $\mathcal{U}_h$ as in Figure \ref{mesh1}. Then, for any triangle $T\in \mathcal{F}_h$,
\begin{equation*}
T~\text{satisfies}~ Maxac(\psi_\alpha),
\end{equation*}
and any quadrilateral $K\in \mathcal{F}_h$,
\begin{equation*}
K~\text{satisfies}~ RDP(N_\alpha,\psi_\alpha).
\end{equation*}
\end{lemma}
\begin{proof}
Combining Theorem \ref{quaderror} with Lemma \ref{rdp} completes the proof.
\end{proof}
\begin{lemma}\label{strip}
Let $T_*$ be the region enclosed by $\Gamma_T$ and $\Gamma_{h,T}$, see Figure \ref{gammaT}. For any $v\in H^1(T_*)$, it holds
\begin{equation}\label{strip1}
\|v\|_{L^2(T_*)}\lesssim h\|v\|_{L^2(\Gamma_T)}+h^2|v|_{H^1(T_*)}.
\end{equation}
\end{lemma}
The proof included below was essentially due to Bramble and King \cite{Bramble1996}.
\begin{proof}
\begin{figure}[H]
\centering
  \includegraphics[width=0.4\textwidth]{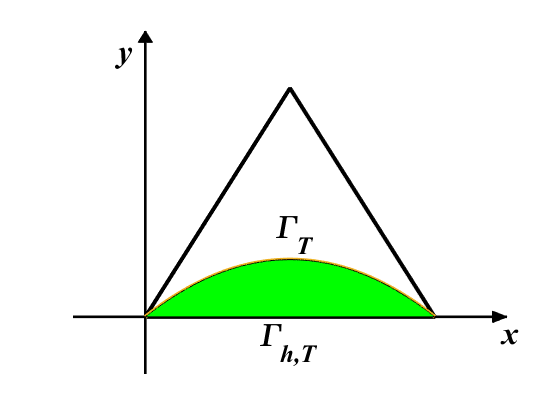}
  \caption{A typical interface element in $\mathcal{F}_h$.}\label{gammaT}
\end{figure}
Assume that $\Gamma_{h,T}$ has its left endpoint at the origin and is given by
\begin{equation*}
\Gamma_{h,T}=\{(x,y)\in T; 0< x\leq \epsilon, y=0\}.
\end{equation*}
Moreover, suppose that $\Gamma_{T}$ can be denoted by
\begin{equation*}
\Gamma_{T}=\{(x,y)\in T; 0< x\leq \epsilon,  y=\eta(x)\},
\end{equation*}
where $\epsilon(\lesssim h)$ is the length of $\Gamma_{h,T}$ and $\eta(x)\in C^{2}(0,\epsilon)$. Since the curvature of $\Gamma$ is bounded, it is known that $\eta(x)\lesssim \epsilon^{2}$ and $\eta^{\prime}(x)\lesssim \epsilon$.
Let $T_*$ be the region enclosed by $\Gamma_T$ and $\Gamma_{h,T}$, by the divergence theorem
\begin{equation*}
\int_{T_*}\nabla{\cdot}\bm{w}dxdy=\int_{\partial T_*}\bm{w}\cdot\bm{n}ds\quad\forall \bm{w}\in H(\mathrm{div};T_*).
\end{equation*}
Let $\bm{w}=(0,yv^{2})^{T}$ with $v\in H^{1}(T_*)$. Then,
\begin{equation*}
\int_{T_*}v^{2}dxdy+\int_{T_*}2yv\frac{\partial v}{\partial y}dxdy=\int_{\Gamma_{T}}yv^{2}(1+(\eta^{\prime}(x))^{2})^{-1/2}ds.
\end{equation*}
Using the Cauchy-Schwarz inequality, it is easy to derive that
\begin{eqnarray*}
\begin{aligned}
\|v\|_{0,T_*}^{2}&\leq C(\|y\|_{0,\infty,\Gamma_{T}}\|v\|_{0,\Gamma_{T}}^{2}+\|y\|_{0,\infty,\Gamma_{T}}\|v\|_{0,T_*}\|\frac{\partial v}{\partial y}\|_{0,T_*})\\
&\leq C\epsilon^{2}\|v\|_{0,\Gamma_{T}}^{2}+C^{2}\epsilon^{4}\|\frac{\partial v}{\partial y}\|_{0,T_*}^{2}+\frac{1}{4}\|v\|_{0,T_*}^{2}.
\end{aligned}
\end{eqnarray*}
Therefore,
\begin{equation*}
\|v\|_{0,T_*}\lesssim h\|v\|_{0,\Gamma_{T}}+h^{2}|v|_{1,T_*}.
\end{equation*}
\end{proof}

The following theorem shows that the generated interface-fitted mesh does not reduce the approximation accuracy in spite of anisotropic elements.
\begin{theorem}\label{interror}
For any $v\in \tilde{H}^2(\Omega_1\cup\Omega_2)$, it holds that
\begin{eqnarray}
\|v-\pi_hv\|_{L^2(\Omega)}\lesssim h^2|v|_{H^2(\Omega_1\cup\Omega_2)},\label{err0}\\
|v-\pi_hv|_{H^1(\Omega)}\lesssim h|v|_{H^2(\Omega_1\cup\Omega_2)}.\label{err1}
\end{eqnarray}
\end{theorem}
\begin{proof}
For non-interface elements, the interpolation error estimate is standard. Assume $K\subset \Omega_{h,i}$ is a general (triangle or quadrilateral) interface element belonging to $\mathcal{F}_h^I$. For simplicity, let $v_i = v|_{\Omega_i}$. Since $\pi_h v= \pi_h Ev_i$, it follows that
\begin{eqnarray*}
\begin{aligned}
|v-\pi_h v|_{L^2(K)}^2
&=|v_i-\pi_h v_i|_{L^2(K_i)}^2+|v_j-\pi_h v_i|_{L^2(K_*)}^2\\
&\lesssim |Ev_i-\pi_h Ev_i|_{L^2(K)}^2+|v_j-Ev_i|_{L^2(K_*)}^2\\
&\lesssim h^4|Ev_i|_{H^2(K)}^2+|v_j-Ev_i|_{L^2(K_*)}^2\\
&\lesssim h^4|Ev_i|_{H^2(K)}^2+h^4|v_j-Ev_i|_{H^1(K_*)}^2\\
&\lesssim h^4(\|Ev_i\|_{H^2(K)}^2+\|v_j\|_{H^2(K_*)}^2),
\end{aligned}
\end{eqnarray*}
where $i,j\in \{1,2\}$ and $j\neq i$. Lemma \ref{strip} and the fact $v_1=v_2$ on $\Gamma$ are used in the fourth step. Summing $K$ over $\mathcal{F}_h^I$, it holds
\begin{align*}
\sum_{K\in \mathcal{F}_h^I}|v-\pi_h v|_{L^2(K)}^2\lesssim h^4 \|v\|_{H^2(\Omega_1\cup\Omega_2)}^2.
\end{align*}
Inequality \eqref{err1} follows by using the same augment.
\end{proof}

\subsection{A prior error estimate}
In the following, the Galerkin approximation of \eqref{weak2} is analyzed in consideration of variational crimes. Recall the interface jump condition in the elliptic interface problem \eqref{elliptic},
\begin{equation*}
[\![u]\!]=q,~~~[\![\beta\frac{\partial u}{\partial n_{\Gamma}}]\!]=g ~~~\text{on}~ \Gamma.
\end{equation*}
%\begin{eqnarray*}
%\left\{
%\begin{aligned}
%&[\![u]\!]=q,~~~\text{on} ~\Gamma\\
%&[\![\beta\frac{\partial u}{\partial n_{\Gamma}}]\!]=g ~~~\text{on}~ \Gamma.
%\end{aligned}
%\right.
%\end{eqnarray*}
\begin{lemma}[\cite{Chen1998}, Lemma 2.2]\label{ggh}
Assume $g\in H^2(\Gamma)$, it holds that
\begin{equation}
|\int_{\Gamma}gv_hds-\int_{\Gamma_h}g_hv_hds|\lesssim h^{3/2}\|g\|_{H^2(\Gamma)}|v_h|_{H^1(\Omega)}\quad\forall v_h\in V_h.
\end{equation}
\end{lemma}
If we don't use the approximation $g_h$ for $g$ in Problem \eqref{weak2}, then $g\in H^{1/2}(\Gamma)$ is enough for the following error analysis.

\begin{lemma}\label{zh}
Assume $q\in H^{3/2}(\Gamma)$, it holds that
\begin{align}
&|z-z_h|_{H^1(\Omega_{2,h})}\lesssim h|z|_{H^2(\Omega_{2,h})},\label{z1}\\
&|\int_{\Omega_{2,h}}\beta_2\nabla z_h\cdot\nabla v_hdx - \int_{\Omega_2}\beta_2\nabla z\cdot\nabla v_hdx|\lesssim \beta_2h\|u\|_{\tilde{H}^2(\Omega_1\cup\Omega_2)}|v_h|_{H^1(\Omega)}\label{z2}.
\end{align}
\end{lemma}
\begin{proof}
The estimate \eqref{z1} is a direct consequence of Theorem \ref{interror}. Using the triangle inequality, it follows that
\begin{align*}
&|\int_{\Omega_{2,h}}\beta_2\nabla z_h\cdot\nabla v_hdx - \int_{\Omega_2}\beta_2\nabla z\cdot\nabla v_hdx|\\
&= |\int_{\Omega_{2,h}}\beta_2 \nabla(z_h-z)\cdot\nabla v_hdx+\int_{\Omega_{2,h}\setminus\Omega_2}\beta_2\nabla z\cdot\nabla v_h dx-\int_{\Omega_{2}\setminus\Omega_{2,h}}\beta_2\nabla z\cdot\nabla v_h dx|\\
&\lesssim \beta_2(|z-z_h|_{H^1(\Omega_{2,h})}
+|z|_{H^1(\Omega_h^I)})|v_h|_{H^1(\Omega)}\\
&\lesssim \beta_2h|u|_{\tilde{H}^2(\Omega_1\cup\Omega_2)}|v_h|_{H^1(\Omega)}.
\end{align*}
\end{proof}

\begin{theorem}
Assume $u\in \tilde{H}^2(\Omega_1\cup\Omega_2)$ is the solution of Problem \eqref{elliptic} and $u_h$ is defined by \eqref{weakrelation}, there exists a constant $C_\beta$ depending on $\beta$ such that
\begin{eqnarray}
|u-u_h|_{H^1(\Omega)}\leq C_\beta h(\|u\|_{\tilde{H}^2(\Omega_1\cup\Omega_2)}+\|g\|_{H^2(\Gamma)}),\label{H1e}\\
\|u-u_h\|_{L^2(\Omega)}\leq C_\beta h^2(\|u\|_{\tilde{H}^2(\Omega_1\cup\Omega_2)}+\|g\|_{H^2(\Gamma)}).\label{L2e}
\end{eqnarray}
\end{theorem}
\begin{proof}
By using Lemma \ref{uh} and the triangle inequality, it yields that
\begin{align*}
|u-u_h|_{H^1(\Omega)}&= |\tilde{u}+z-\tilde{u}_h- z_{h}|_{H^1(\Omega)}\\
&\leq |\tilde{u}-\tilde{u}_h|_{H^1(\Omega)}+|z- z_h|_{H^1(\Omega)}.
\end{align*}
The term $|z-z_h|_{H^1(\Omega)}$ is already analyzed in Lemma \ref{zh}. By the standard analysis, it follows that
\begin{equation*}
|\tilde{u}-\tilde{u}_h|_{H^1(\Omega)}\leq |\tilde{u}-\pi_h \tilde{u}|_{H^1(\Omega)}+\min\{\beta_1^{-1},\beta_2^{-1}\}\sup\limits_{v_h\in V_h}\frac{a_h(\tilde{u}-\tilde{u}_h,v_h)}{|v_h|_{H^1(\Omega)}}.
\end{equation*}
The first term is the interpolation error proved in Theorem \ref{interror}. The second term is the consistence error, and it is straightforward to show that
\begin{equation*}
a_h(\tilde{u}-\tilde{u}_h,v_h)= (a_h(\tilde{u},v_h)-a(\tilde{u},v_h))+(F(v_h)-\tilde{F}_h(v_h)).
\end{equation*}
Then, Lemma \ref{strip} and Lemma \ref{ggh} imply that
\begin{eqnarray*}
\begin{aligned}
a_h(\tilde{u},v_h)-a(\tilde{u},v_h)
&=\int_{\Omega_{1}\cap\Omega_{h,2}}(\beta_2-\beta_1)\nabla \tilde{u}\cdot\nabla v_hdx
+\int_{\Omega_{2}\cap\Omega_{h,1}}(\beta_1-\beta_2)\nabla \tilde{u}\cdot\nabla v_hdx\\
&\lesssim |\beta_1-\beta_2| |\tilde{u}|_{H^1(\Omega_h^I)}|v_h|_{H^1(\Omega)}\\
&\lesssim |\beta_1-\beta_2| h\|u\|_{\tilde{H}^2(\Omega_1\cup\Omega_2)}|v_h|_{H^1(\Omega)},
\end{aligned}
\end{eqnarray*}
and
\begin{align*}
F(v_h)-\tilde{F}_h(v_h)&= (\int_\Gamma gv_hds- \int_\Gamma g_hv_hds) + (\int_{\Omega_{2,h}}\beta_2\nabla z_h\cdot\nabla v_hdx - \int_{\Omega_2}\beta_2\nabla z\cdot\nabla v_hdx)\\
&\lesssim h(\|g\|_{H^{3/2}(\Gamma)}+\beta_2\|u\|_{\tilde{H}^2(\Omega_1\cup\Omega_2)})|v_h|_{H^1(\Omega)}.
\end{align*}
The desired result \eqref{H1e} then follows. And the $L^2$-norm error estimate \eqref{L2e} can be proved by using the dual argument.
\end{proof}

\section{A Multi-grid Iterative Method}\label{multigrid}

\subsection{Spaces decomposition}
Suppose $\mathcal{U}_0$ is a quasi-uniform and regular triangulation of $\Omega$ with the mesh size $h_0$. Let $\mathcal{U}_l$, $1\leq l\leq L$, be obtained from $\mathcal{U}_{l-1}$ via a ``regular" subdivision: edge midpoints in $\mathcal{U}_{l-1}$ are connected by new edges to form $\mathcal{U}_{l}$.

For any integer $L\geq 2$, let $\mathcal{F}_L$ be an interface-fitted mesh generated from $\mathcal{U}_L$. In order to derive an optimal multi-grid method for the interface problem, we construct a sequence of nested triangulations $\{\mathcal{F}_l\}_{l=0}^L$ as follows, see Figure \ref{seqT} for an illustration. The blue line is an interface, the red triangles, yellow squares and black points denotes degree of freedoms on different levels.
\begin{algorithm}[H]
\caption{Generate an interface adaptive mesh $\mathcal{F}_l~(0\leq l\leq L-1)$.}
\begin{algorithmic}[]\label{vcycle}
\STATE (1) Generate an interface-unfitted quasi-uniform mesh $\mathcal{U}_l$ with its mesh size equal $h_l$. Denote $\mathcal{U}_{l,0}=\mathcal{U}_l$;\\
\STATE (2) If $l<L$, refine $T\in \mathcal{U}_{l,j}$ via ``regular" subdivision if $T$ is an interface element or is a neighborhood of an interface element to get $\mathcal{U}_{l,j+1}$. Repeat step (2) to get $\mathcal{U}_{l,L-l}$;\\
\STATE (3) Connect intersected points of $\Gamma$ and $\mathcal{U}_{l,L-l}$ to generate an interface-fitted grid $\mathcal{F}_l$;\\
\end{algorithmic}
\end{algorithm}
\vspace{-1cm}
\begin{figure}[H]\centering
\includegraphics[width=0.6\textwidth]{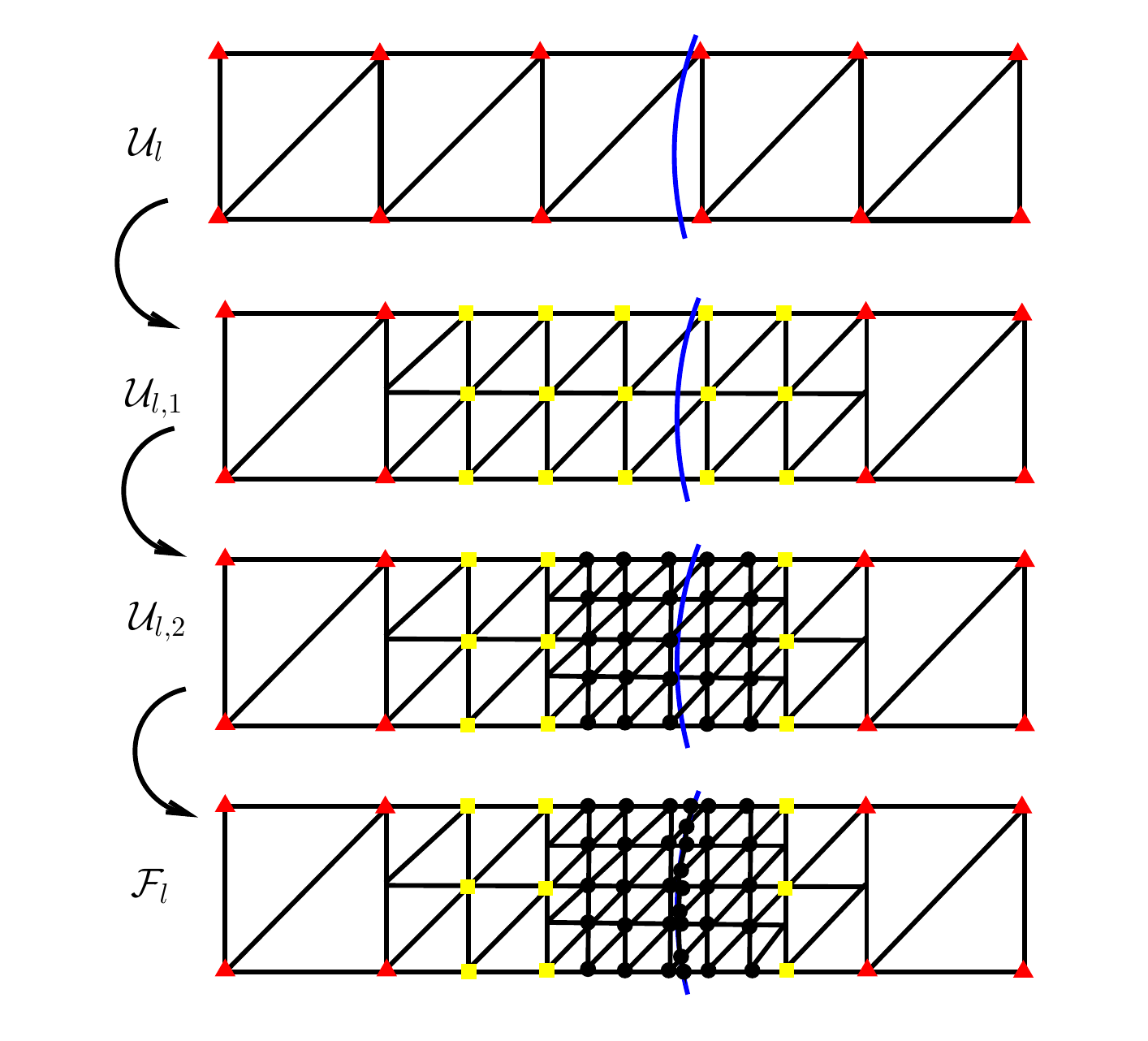}
\caption{An illustration for generating $\mathcal{F}_0$ with $L=2$.}
\label{seqT}
\end{figure}

Recall that $\mathcal{U}_{l}^I$ is the collection of interface elements in $\mathcal{U}_{l}$. Let $\mathcal{U}_{l,\Gamma}$ be an extension of $\mathcal{U}_{l}^I$ which also contains all neighbor elements of interface elements, i.e.,
\begin{equation*}
\mathcal{U}_{l,\Gamma} = \{T\in \mathcal{U}_{l}; \exists \tilde{T}\in \mathcal{U}_{l}^I, s.t.~~ \partial T\cap \partial \tilde{T}\neq \emptyset\}.
\end{equation*}
The corresponding region for the above element collection is denoted by $\Omega_{l,\Gamma}$ (The gray-painted area in Figure \ref{omega}). Define $\omega_{i,l} = \Omega_i\setminus \Omega_{l,\Gamma}$ for $i=1,2$. Consequently, the domain $\Omega$ admits the following division on the $l$-th level (see Figure \ref{omega})
\begin{equation}
\Omega = \bar{\Omega}_{l,\Gamma}\cup\omega_{l,1}\cup\omega_{l,2}.
\end{equation}
%By definition, the following relation holds
%\begin{equation}
%\Omega_{L,\Gamma}\subset\Omega_{L-1,\Gamma}\subset\cdots
%\subset\Omega_{0,\Gamma}\subset\Omega_{-1,\Gamma}:=\Omega.
%\end{equation}
%The interface-unfitted mesh $\mathcal{U}_l$ can be decomposed as follows
%\begin{align}
%\mathcal{U}_l = \mathcal{U}_{l,\Gamma}^{l-1}\oplus \mathcal{U}_{l,\omega}^{l-1},
%\end{align}
%where
%\begin{align}
%\mathcal{U}_{l,\Gamma}^{l-1}&=\{T\in \mathcal{U}_l; T\subset \Omega_{l-1,\Gamma}\},\\ \mathcal{U}_{l,\omega}^{l-1}&=\{T\in \mathcal{U}_l; T\subset \omega_{l-1,1}\cup\omega_{l-1,2}\}.
%\end{align}
Let $W_l=span\{\phi_{l,i}\}_{i=1}^{n_l}$ be the $P_1$ conforming finite element space defined on $\mathcal{U}_l$, where $\{\phi_{l,i}\}_{i=1}^{n_l}$ is the nature nodal basis such that $\phi_{l,i}(x_{l,j})=\delta_{i,j}$ for each non-Dirichlet boundary node $x_{l,j}$. Decompose the space $W_l$ according to the division of the domain $\Omega$ on $(l-1)$-th level
\begin{equation*}
W_l = W_{l,\Gamma}\oplus W_{l,\omega},
\end{equation*}
where
\begin{align*}
W_{l,\Gamma}&=span\{\phi_{l,i}; x_{l,i}\in\Omega_{l-1,\Gamma}\},\\
W_{l,\omega}&=span\{\phi_{l,i}; x_{l,i}\in\Omega\setminus\Omega_{l-1,\Gamma}\}.
\end{align*}
Define $n_{l,\Gamma}=\text{dim}(W_{l,\Gamma})$ and $n_{l,\omega}=\text{dim}(W_{l,\omega})=n_l-n_{l,\Gamma}$. Without loss of generality, assume $W_{l,\omega}= span\{\phi_{l,i}\}_{i=1}^{n_{l,\omega}}$ and $W_{l,\Gamma}=span\{\phi_{l,i}\}_{i=n_{l,\omega}+1}^{n_l}$. Denote $\phi_{l,0}=\sum_{i=n_{l,\omega}+1}^{n_l}\phi_{l,i}$ and $\omega_{l,\Gamma}=\{(x,y)\in \Omega; \phi_{l,0}(x,y)\equiv 1\}$ (the green region in Figure \ref{omega}). In another word, $\omega_{l,\Gamma}$ is obtained by shrinking one element width inwards from $\Omega_{l-1,\Gamma}$.
%Notice that $\{\phi_{l,j}\}_{j=0}^{n_l}$ can be viewed as a partition of unity.
\begin{figure}\centering
\includegraphics[width=0.6\textwidth]{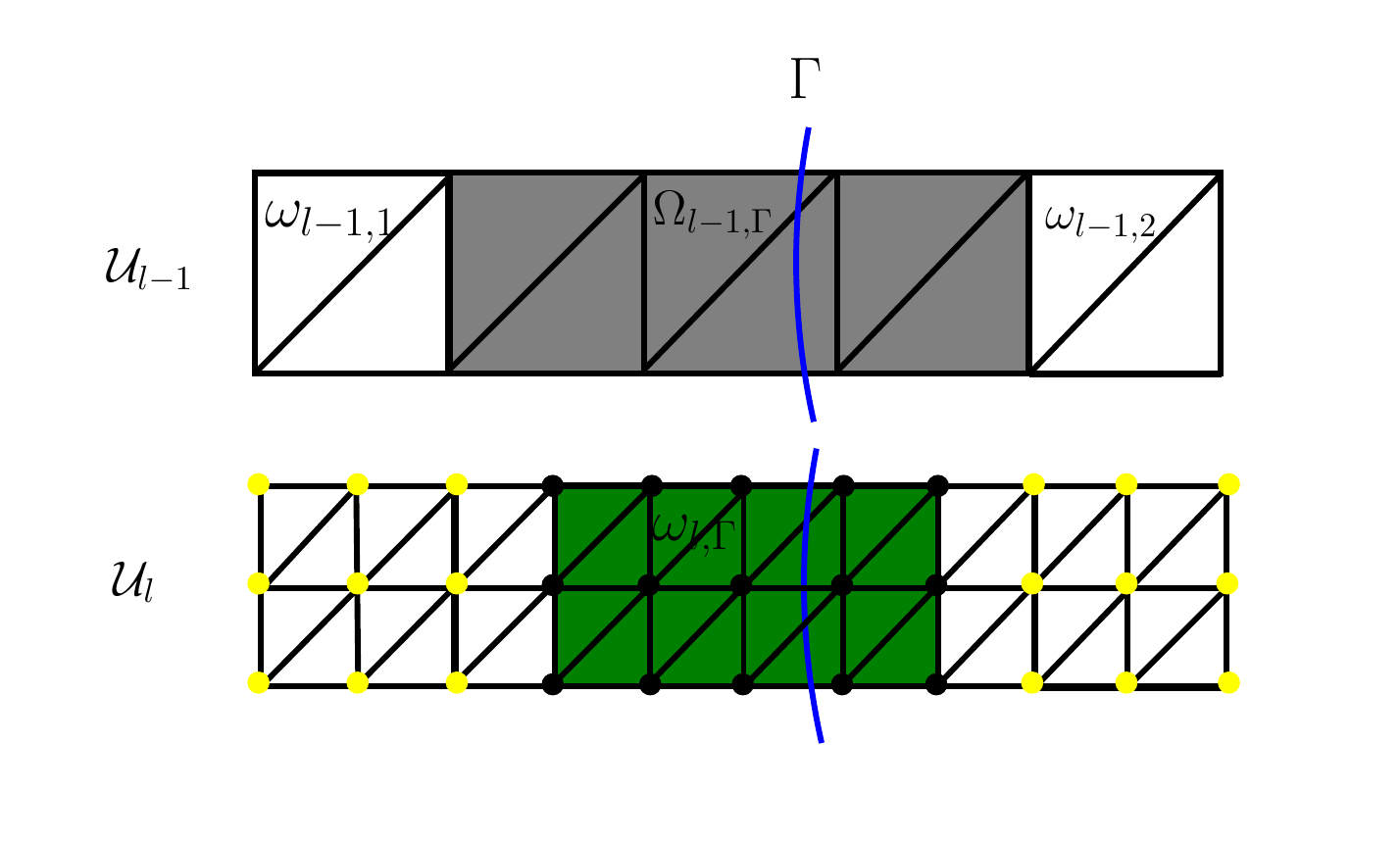}
\caption{An illustration for the space decomposition of $W_l$: black dots represent the degree of freedom of $W_{l,\Gamma}$; yellow dots represent the degree of freedom of $W_{l,\omega}$.}
\label{omega}
\end{figure}

Let $V_L$ be the local anisotropic finite element space defined on $\mathcal{F}_L$ as in \eqref{aspace}, and
\begin{equation*}
V_{L,0}=\{v_h\in V_L; supp(v_h)\subset\Omega_{L-1,\Gamma}\}.
\end{equation*}
Then the coarse space $V_l(1\leq l\leq L-1)$ associated with the interface adaptive mesh $\mathcal{F}_l$ is defined as follows
\begin{equation}
V_l=V_{l,0}\oplus V_{l,\omega},
\end{equation}
where
\begin{eqnarray*}
&&V_{l,0}=\sum_{j=l}^{L-1}W_{j,\Gamma}
+V_{L,0},\\
&&V_{l,\omega}=W_{l,\omega},
\end{eqnarray*}
and $V_0=\sum_{j=0}^{L-1}W_{j,\Gamma}+V_{L,0}$. Obviously, the inclusion relation holds
\begin{equation*}
V_0\subset V_1\subset\cdots\subset V_{L-1}\subset V_L.
\end{equation*}
For each space $V_{l,\omega}$, it can be further decomposed into micro pieces
\begin{equation*}
V_{l,\omega}=\sum_{j=1}^{n_{l,\omega}}V_{l,j}=\sum_{j=1}^{n_{l,\omega}}span\{\phi_{l,j}\}.
\end{equation*}
Therefore, the decomposition of $V_L$ can be written as follows
\begin{equation}\label{spdecomp}
V_L=\sum_{l=0}^L V_l=V_0+\sum_{l=1}^L (V_{l,0}+\sum_{j=1}^{n_{l,\omega}}V_{l,j}).
\end{equation}
\begin{remark}
The key idea of the grid coarsening is to keep elements near the interface on each level without any coarsening, and coarsen elements far from the interface in a standard way. Since the number of degrees of freedom in elements near the interface ($V_{l,0}$) is $O(h_l^{-1})$, we can solve the residual equation on $V_{l,0}$ with an exact solver.
\end{remark}
\subsection{The multigrid algorithm}
For convenience, let $a_L(\cdot,\cdot)$ be short for $a_{h_L}(\cdot,\cdot)$ and $\Omega_{i,L}$ be short for $\Omega_{i,h_L}$, $i=1,2$. On each level $0\leq l\leq L$, the operator $A_l:V_l\rightarrow V_l$ is defined by
\begin{equation*}
(A_lv,w) = a_L(v,w)\quad\forall v,w\in V_l.
\end{equation*}
Then the weak formulation \eqref{weak2} is equivalent to the following operator equation
\begin{equation}\label{Oeq}
A_Lu_L = F_L,
\end{equation}
where $F_L\in V_L$ such that $(F_L,v)=\bar{F}_{h_L}(v)$, $\forall v\in V_L$. For $v\in H^1(\Omega)$, define the following weighted $L^2$-norm and $H^1$-norm
\begin{align*}
\normmm{v}_{0} = (\|\beta_1^{1/2}v\|_{L^2(\Omega_{1,L})}^2 + \|\beta_2^{1/2}v\|_{L^2(\Omega_{2,L})}^2)^{1/2},\\
\normmm{v}_{1} = (|\beta_1^{1/2}v|_{H^1(\Omega_{1,L})}^2 + |\beta_2^{1/2}v|_{H^1(\Omega_{2,L})}^2)^{1/2}.
\end{align*}
%and the operator norm
%\begin{equation}
%\|T\|_A=\sup_{\normmm{v}_1=1}\normmm{T v}_1,\quad\forall T\in V_L^*.
%\end{equation}
And for any set $D\subset \Omega$, define
\begin{align*}
\normmm{v}_{0,D} = (\|\beta_1^{1/2}v\|_{L^2(\Omega_{1,L}\cap D)}^2 + \|\beta_2^{1/2}v\|_{L^2(\Omega_{2,L}\cap D)}^2)^{1/2},\\
\normmm{v}_{1,D} = (|\beta_1^{1/2}v|_{H^1(\Omega_{1,L}\cap D)}^2 + |\beta_2^{1/2}v|_{H^1(\Omega_{2,L}\cap D)}^2)^{1/2}.
\end{align*}

Let $Q_l:L^2(\Omega)\rightarrow V_l$ be the standard orthogonal $L^2$ projection defined by
\begin{equation*}
(Q_l u, w)=(u, w)~~\forall w\in V_l.
\end{equation*}
For $j=0,1,\cdots,n_{l,\omega}$, define $P_{l,j}:V_L\rightarrow V_{l,j}$ by
\begin{equation*}
a_L(P_{l,j}u,w)=a_L(u,w)\quad\forall w\in V_{l,j},
\end{equation*}
and $P_0:V_L\rightarrow V_{0}$ is defined by
\begin{equation*}
a_L(P_0u,w)=a_L(u,w)\quad\forall w\in V_{0}.
\end{equation*}
Denoting $W_l(\Omega\setminus\Omega_{l-1,\Gamma})=\{v|_{\Omega\setminus\Omega_{l-1,\Gamma}}; v\in W_l\}$, let $Q_{l,\beta}^{\omega}: L^2(\Omega\setminus\Omega_{l-1,\Gamma})\rightarrow W_l(\Omega\setminus\Omega_{l-1,\Gamma})$ be the $L^2$ projection defined as follows
\begin{equation*}
(\beta Q_{l,\beta}^{\omega} v, w)_{\Omega\setminus\Omega_{l-1,\Gamma}}=(\beta u, w)_{\Omega\setminus\Omega_{l-1,\Gamma}}~~\forall w\in W_l(\Omega\setminus\Omega_{l-1,\Gamma}).
\end{equation*}
Let $Q_{l,\beta}v\in W_l$ be such that
\begin{equation*}
Q_{l,\beta}v=\left\{
\begin{aligned}
&0,~~~~~~~\text{in}~ \omega_{l,\Gamma},\\
&Q_{l,\beta}^{\omega}v,~~\text{in}~ \Omega\setminus\Omega_{l-1,\Gamma}.
\end{aligned}
\right.
\end{equation*}
Since the region $\Omega_{l-1,\Gamma}\setminus\omega_{l,\Gamma}$ is a band of one triangular element width (see Figure \ref{deofQ}), $Q_{l,\beta}v$ is a well-defined and unique function in $W_l$. By simultaneous approximation property, it holds
\begin{equation*}
\normmm{v-Q_{l,\beta}v}_{0,\Omega\setminus\Omega_{l-1,\Gamma}}\lesssim h_l\normmm{v}_{1,\Omega\setminus\Omega_{l-1,\Gamma}}.
\end{equation*}
\begin{figure}[H]
\vspace{-0.5cm}
\centering
  \includegraphics[width=0.6\textwidth]{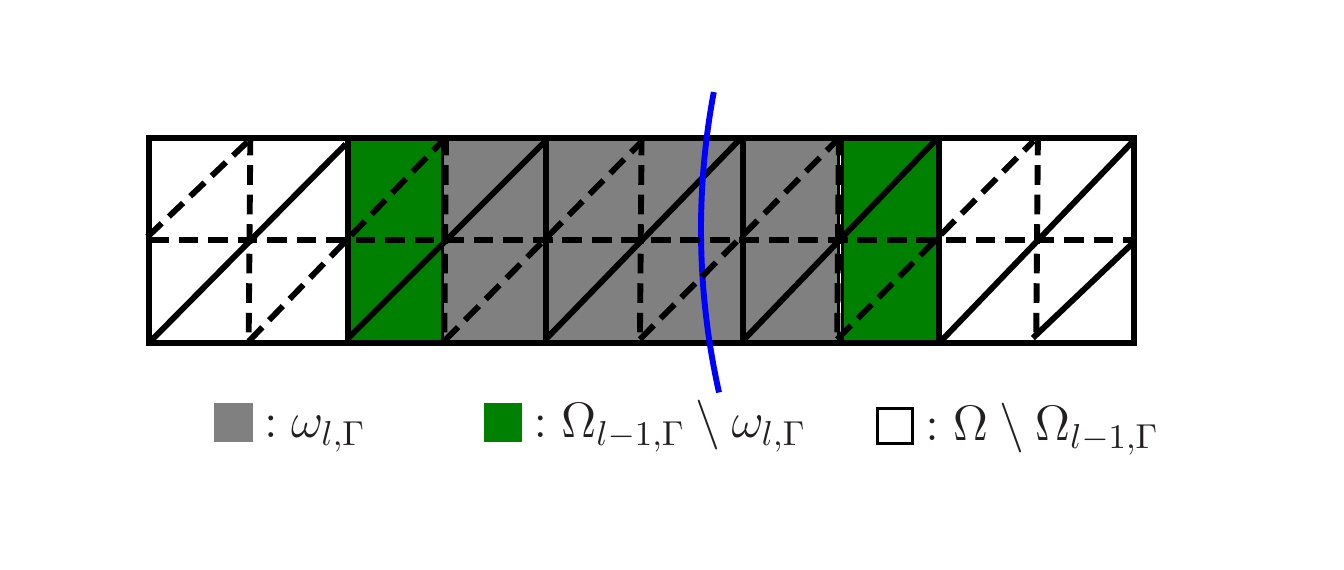}
  \vspace{-1cm}
  \caption{An illustration for the definition of $Q_{l,\beta}$.}\label{deofQ}
\end{figure}
%$P_0:H^1_0(\Omega)\rightarrow V_{0}$ be defined by
%\begin{equation*}
%a_L(P_0u,w)=a_L(u,w)\quad\forall w\in V_{0}.
%\end{equation*}
The block Gauss-Seidel smoother $R_l:V_l\rightarrow V_l$ is defined by
\begin{equation*}
R_l=(I-\prod_{i=0}^{n_{l,\omega}}(I-P_{l,i}))A_l^{-1}.
\end{equation*}
\begin{remark}
This block Gauss-Seidel smoother $R_l$ means that
\begin{enumerate}
      \item do subspace correction on $V_{l,0}$ with an exact solver.
      \item do subspace correction on $V_{l,\omega}$ with the point Gauss-Seidle smoother.
\end{enumerate}
\end{remark}
The operator $B_L:V_L\rightarrow V_L$ is defined recursively as follows
\begin{algorithm}[H]
\caption{(V-cycle)Let $B_0=A_0^{-1}$, for $F_L\in V_L$, define $B_LF_L=u^{(3)}$.}
\begin{algorithmic}[]\label{vcycle}
\STATE (1)Presmoothing: $u^{(1)}=R_L F_L$;\\
\STATE (2)Correction: $u^{(2)}=u^{(1)}+B_{L-1}Q_{L-1}(F_L-A_Lu^{(1)})$;\\
\STATE (3)Postsmoothing:$u^{(3)}=u^{(2)}+R_L^\ast(F_L-A_L u^{(2)})$;\\
\end{algorithmic}
\end{algorithm}
It is straightforward to show that (see \cite{Xu2002} for details)
\begin{equation}
I-B_LA_L=\left((I-P_0)\prod_{l=1}^L\prod_{i=0}^{n_{l,\omega}}(I-P_{l,i})\right)^\ast
\left((I-P_0)\prod_{l=1}^L\prod_{i=0}^{n_{l,\omega}}(I-P_{l,i})\right).
\end{equation}
\subsection{V-cycle convergence for the $L^{th}$ level iteration}
Let $\pi_l: C_0(\Omega)\rightarrow V_l$ be the nodal interpolation operator. Define $\Pi_l:V_L\rightarrow V_l$ by
\begin{equation}
\Pi_lv=Q_{l,\beta}v+\pi_{l}((\Pi_{l+1}v-Q_{l,\beta}v)\phi_{l,0})\quad \text{for}~ l=L-1,L-2,\cdots,0,
\end{equation}
where $\Pi_L=I$. By the definition of $\phi_{l,0}$, it holds
\begin{eqnarray}\label{beta}
\Pi_lv=\left\{
\begin{aligned}
&\Pi_{l+1}v,~~\text{in}~ \omega_{l,\Gamma},\\
&Q_{l,\beta}v,~~\text{in}~ \Omega\setminus\Omega_{l-1,\Gamma}.
\end{aligned}
\right.
\end{eqnarray}
\begin{figure}[H]
\vspace{-0.5cm}
\centering
  \includegraphics[width=0.4\textwidth]{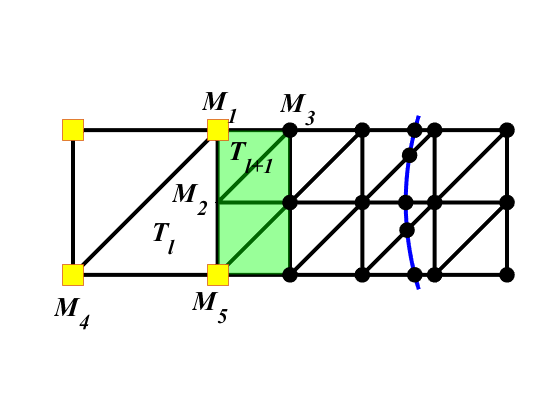}
  \vspace{-1cm}
  \caption{The green-painted area: $\Omega_{l-1,\Gamma}\setminus\omega_{l,\Gamma}$.}\label{patch}
\end{figure}
\begin{lemma}\label{pie1}
For $0\leq l\leq L-1$, it holds that
\begin{equation}\label{pierr1}
\normmm{(\Pi_{l+1}-\Pi_l) v}_{0}+h_l\normmm{(\Pi_{l+1}-\Pi_l) v}_{1}\lesssim h_l\normmm{v}_{1}\quad \forall v\in V_L.
\end{equation}
\end{lemma}
\begin{proof}
By the definition of $\Pi_l$, it has
\begin{eqnarray*}
(\Pi_{l+1}-\Pi_l) v=\left\{
\begin{aligned}
&0,~&\text{in}~&\omega_{l,\Gamma},\\
&(Q_{l+1,\beta}-Q_{l,\beta})v,~&\text{in}~& \Omega\setminus\Omega_{l-1,\Gamma}.
\end{aligned}
\right.
\end{eqnarray*}
Since
\begin{eqnarray*}
\begin{aligned}
&\normmm{(Q_{l+1,\beta}-Q_{l,\beta})v}_{0,\Omega\setminus\Omega_{l-1,\Gamma}}^2
+h_l^2\normmm{(Q_{l+1,\beta}-Q_{l,\beta})v}_{1,\Omega\setminus\Omega_{l-1,\Gamma}}^2\\
&\leq \sum\limits_{j=l}^{l+1}(\normmm{v-Q_{j,\beta}v}_{0,\Omega\setminus\Omega_{l-1,\Gamma}}^2
+h_l^2\normmm{v-Q_{j,\beta}v}_{1,\Omega\setminus\Omega_{l-1,\Gamma}}^2)\\
&\lesssim h_l^2 \normmm{v}_{1}^2,
\end{aligned}
\end{eqnarray*}
it only needs to estimate $\normmm{(Q_{l+1,\beta}-Q_{l,\beta})v}_{0,\Omega_{l-1,\Gamma}\setminus \omega_{l,\Gamma}}$.
Without loss of generality, assume $T_{l+1}\subset \Omega_{l-1,\Gamma}\setminus \omega_{l,\Gamma}$ as in Figure \ref{patch}. It is obvious that $T_{l+1}\subset \Omega\setminus \Omega_{l,\Gamma}$, therefore $\Pi_{l+1} v|_{T_{l+1}}= Q_{l+1,\beta}v$. A direct calculation shows that
\begin{eqnarray*}
\begin{aligned}
\normmm{(\Pi_{l+1}-\Pi_l)v}_{0,T_{l+1}}^2&\lesssim h_{l+1}^2\sum_{1\leq i\leq 2}\beta(M_1)(Q_{l+1,\beta}v(M_i)-Q_{l,\beta}v(M_i))^2\\
&\lesssim h_{l+1}\beta(M_1)\|Q_{l+1,\beta}v-Q_{l,\beta}v\|_{0,M_1 M_2}^2\\
&\lesssim h_{l+1}\beta(M_1)\|Q_{l+1,\beta}v-Q_{l,\beta}v\|_{0,M_1 M_5}^2\\
&\lesssim \normmm{Q_{l+1,\beta}v-Q_{l,\beta}v}_{0,T_{l}}^2
+h_{l+1}^2\normmm{Q_{l+1,\beta}v-Q_{l,\beta}v}_{0,T_{l}}^2\\
&\lesssim h_{l}^2\normmm{v}_{1,T_l}^2.
\end{aligned}
\end{eqnarray*}
Similarly, it is easy to derive that $\normmm{(\Pi_{l+1}-\Pi_l)v}_{1,T_{l+1}}^2 \lesssim \normmm{v}_{1,T_l}^2$.
Summing $T_{l+1}$ over $\Omega_{l,\Gamma}\setminus \omega_{l,\Gamma}$, it yields
\begin{equation*}
\normmm{(\Pi_{l+1}-\Pi_l)v}_{0,\Omega_{l-1,\Gamma}\setminus \omega_{l,\Gamma}}^2+h_l^2\normmm{(\Pi_{l+1}-\Pi_l)v}_{1,\Omega_{l-1,\Gamma}\setminus \omega_{l,\Gamma}}^2
\lesssim h_{l}^2\normmm{v}_{1}^2.
\end{equation*}
This completes the proof.
\end{proof}

\begin{lemma}\label{pie2}
For $0\leq l\leq L$, it holds
\begin{equation}
\normmm{v-\Pi_l v}_0+h_l\normmm{v-\Pi_l v}_1\leq Ch_l\normmm{v}_1\quad \forall v\in V_L.
\end{equation}
\end{lemma}
\begin{proof}
By Lemma \ref{pie1} and the triangle inequality, it's easy to derive that
\begin{eqnarray*}
\begin{aligned}
&\normmm{v-\Pi_l v}_0+h_l\normmm{v-\Pi_l v}_1\\
&\leq \sum\limits_{r=l}^{L-1}(\normmm{\pi_{r+1}v-\Pi_r v}_0
+h_r\normmm{\Pi_{r+1}v-\Pi_r v}_1)\\
&\lesssim h_l \normmm{v}_1.
\end{aligned}
\end{eqnarray*}
\end{proof}
According to the space decomposition \eqref{spdecomp}, for any $v\in V_L$, do the following decomposition
\begin{eqnarray}\label{fundecomp}
v=v_0+\sum_{l=1}^L(v_{l,0}+\sum\limits_{i=1}^{n_{l,\omega}}v_{l,i}),
\end{eqnarray}
where $v_0=\Pi_0 v\in V_0$, and $v_{l,i}=\pi_l((\Pi_{l}v-\Pi_{l-1}v)\phi_{l,i})\in V_{l,i}$ for $i=0,1,\cdots,n_{l,\omega}$.
\begin{lemma}
The V-cycle algorithm \eqref{vcycle} has the following convergence rate estimate
\begin{equation}
\|I-B_LA_L\|_{A_L}\leq 1-\frac{1}{1+C|\log h_L|}.
\end{equation}
\end{lemma}
The proof included below mainly follows \cite{Xu2008}, specific details are still given for completeness.
\begin{proof}
By X-Z identity (see \cite{Xu2002}), it holds
\begin{equation*}
\|I-B_LA_L\|_{A_L}\leq 1-\frac{1}{1+c_0},
\end{equation*}
with
\begin{equation*}
c_0=\sup_{\normmm{v}_{1}=1}\inf_{v=v_0+\sum_{l=1}^L\sum\limits_{i=0}^{n_{l,\omega}}v_{l,i}}c(v),
\end{equation*}
where
\begin{equation*}
c(v)=\normmm{P_0(v-v_0)}_{1}^2
+\sum\limits_{l=1}^L\sum_{i=0}^{n_{l,\omega}}\normmm{P_{l,i}
\sum_{(k,j)>(l,i)}v_{k,j}}_{1}^2.
\end{equation*}
Let $\Delta_{l,i}=supp(\phi_{l,i})$ and notice
\begin{equation*}
\sum_{(k,j)>(l,i)}v_{k,j}=(v-\Pi_l v)+\sum\limits_{j=i+1}^{n_{l,\omega}}v_{l,j},
\end{equation*}
\begin{eqnarray*}
\begin{aligned}
\sum\limits_{l=1}^L\sum_{i=0}^{n_{l,\omega}}\normmm{P_{l,i}\sum_{(k,j)>(l,i)}v_{k,j}}_{1}^2
&\lesssim \sum\limits_{l=1}^L\sum_{i=0}^{n_{l,\omega}}(\normmm{P_{l,i}(v-\Pi_lv)}_{1}^2
+\normmm{P_{l,i}\sum_{j=i+1}^{n_{l,\omega}}v_{l,j}}_{1}^2)\\
&\lesssim \sum\limits_{l=1}^L\sum_{i=0}^{n_{l,\omega}}(\normmm{v-\Pi_lv}_{1,\Delta_{l,i}}^2
+\sum_{j=i+1}^{n_{l,\omega}}\normmm{v_{l,j}}_{1,\Delta_{l,i}}^2).
\end{aligned}
\end{eqnarray*}
Observe that
\begin{equation*}
\sum\limits_{l=1}^L\sum_{i=0}^{n_{l,\omega}}\normmm{v-\Pi_lv}_{1,\Delta_{l,i}}^2
=\sum\limits_{l=1}^L\normmm{v-\Pi_lv}_{1,\Omega}^2\lesssim |\log h_L|\normmm{v}_{1,\Omega}^2,
\end{equation*}
and
\begin{eqnarray*}
\begin{aligned}
\sum\limits_{l=1}^L\sum_{i=0}^{n_{l,\omega}}\sum_{j=i+1}^{n_{l,\omega}}\normmm{v_{l,j}}_{1,\Delta_{l,i}}^2
&\lesssim\sum\limits_{l=1}^L(\normmm{v_{l,0}}_{1,\Delta_{l,0}}^2+\sum_{i=1}^{n_{l,\omega}}
\sum_{j=i+1}^{n_{l,\omega}}h_l^{-2}\normmm{v_{l,j}}_{0,\Delta_{l,i}}^2)\\
&\lesssim \sum\limits_{l=1}^L(\normmm{\pi_l((\Pi_lv-\Pi_{l-1}v)\phi_{l,0})}_{1,\Delta_{l,0}}^2
+\sum_{i=1}^{n_{l,\omega}}\sum_{j=i+1}^{n_{l,\omega}}h_l^{-2}\normmm{v_{l,j}}_{0,\Delta_{l,i}}^2)\\
&\lesssim \sum\limits_{l=1}^L(\normmm{\Pi_lv-\Pi_{l-1}v}_{1,\Delta_{l,0}}^2
+\sum_{i=0}^{n_{l,\omega}}h_l^{-2}\normmm{\Pi_lv-\Pi_{l-1}v}_{0,\Delta_{l,i}}^2)\\
&\lesssim |\log h_L|\normmm{v}_{1}^2.
\end{aligned}
\end{eqnarray*}
Combining the above inequalities concludes the proof.
\end{proof}

\section{Numerical Examples}
\label{sec:example}
In this section, we present several numerical examples to show the performance of our methods. Particular attention will be paid on verifying its high order convergence and examining its robustness in dealing with low regularity solutions and complex geometries. The computational domain is the rectangle $-1\leq x,y\leq 1$, and the interface is denoted by a levelset function $\phi(x,y)$, i.e.,
\begin{align*}
&\Omega=\{(x,y)\in \mathbb{R}^{2}; -1\leq x,y\leq 1\},\\
&\Omega_{1}=\{(x,y)\in \Omega; \phi(x,y)>0\},\\
&\Omega_{2}=\{(x,y)\in \Omega; \phi(x,y)<0\}.
\end{align*}
We test the multigrid algorithm \ref{vcycle} with these examples, the initial guess is $\bm{0}$, and the stopping criterion is the $l^2$ norm of the relative residual being smaller than exp(-20).
\subsection{Example 1}
The interface is a circle centered at the origin with radius $r$, i.e.,
\begin{equation*}
\phi(x,y)=x^2+y^2-r^2.
\end{equation*}
The exact solution is chosen as follows
\begin{equation*}
u = \frac{1}{\beta}\phi(x,y)\sin(\pi x)\sin(\pi y).
\end{equation*}
We test the local anisotropic FEM for the second order elliptic interface problem \eqref{elliptic} whose exact solutions are defined as above and whose coefficient jump ratio $\beta_{1}/\beta_{2}=10^4,10^2,10^{-2},10^{-4}$. Numerical results are shown in Tables \ref{tab1}-\ref{tab4}, illustrating that the convergence rates are optimal in $H^1$-norm and $L^2$-norm. Figure \ref{uh112} illustrates that our method allows discontinuity of the gradient of the solution $u$ on the interface.

\vspace{1cm}
\begin{minipage}{\textwidth}
 \begin{minipage}[t]{0.5\textwidth}
  \centering
     \makeatletter\def\@captype{table}\makeatother
\captionsetup{font={small}}
\caption{\label{tab1}
Finite element errors  for {\bf Example 1} with $\beta_{1}=10^{4}$, $\beta_{2}=1$.}
\small
\begin{tabular}{|c|c|c|c|c|}
\hline
$\frac{1}{h}$          &     $\|u-u_{h}\|_{0,\Omega}$   & order&  $|u-u_{h}|_{1,\Omega}$  &   order    \\
\hline
 32&   1.3399e-03  &          &  3.3520e-02  &         \\
\hline
 64&   3.6122e-04  &  1.8911 &  1.7466e-02  &  0.9404\\
\hline
 128&  9.0503e-05  &  1.9968 &  8.8375e-03  &  0.9828\\
\hline
 256&  2.2666e-05  &  1.9974 &  4.4497e-03  &  0.9899\\
\hline
 512&  5.6388e-06  &  2.0070 &  2.2724e-03  &  0.9694\\
\hline
\end{tabular}
  \end{minipage}
  \begin{minipage}[t]{0.5\textwidth}
   \centering
        \makeatletter\def\@captype{table}\makeatother
\captionsetup{font={small}}
\caption{\label{tab2}
Finite element errors  for {\bf Example 1} with $\beta_{1}=10^{2}$, $\beta_{2}=1$.}
\small
\begin{tabular}{|c|c|c|c|c|}
\hline
$\frac{1}{h}$          &     $\|u-u_{h}\|_{0,\Omega}$   & order&  $|u-u_{h}|_{1,\Omega}$  &   order    \\
\hline
 32&   1.3415e-03  &         &  3.3523e-02  &         \\
\hline
 64&   3.6107e-04  &  1.8935 &  1.7467e-02  &  0.9404\\
\hline
 128&  9.0456e-05  &  1.9969 &  8.8399e-03  &  0.9825\\
\hline
 256&  2.2638e-05  &  1.9984 &  4.4511e-03  &  0.9898\\
\hline
 512&  5.6057e-06  &  2.0138 &  2.2727e-03  &  0.9697\\
\hline
\end{tabular}
   \end{minipage}
\end{minipage}

\vspace{1cm}
\begin{minipage}{\textwidth}
 \begin{minipage}[t]{0.5\textwidth}
  \centering
     \makeatletter\def\@captype{table}\makeatother
\captionsetup{font={small}}
\caption{\label{tab3}
Finite element errors  for {\bf Example 1} with $\beta_{1}=1$, $\beta_{2}=10^{2}$.}
\small
\begin{tabular}{|c|c|c|c|c|}
\hline
$\frac{1}{h}$          &     $\|u-u_{h}\|_{0,\Omega}$   & order&  $|u-u_{h}|_{1,\Omega}$  &   order    \\
\hline
 32&   3.9442e-03  &         &  9.7003e-02  &         \\
\hline
 64&   9.9666e-04  &  1.9845 &  4.8823e-02  &  0.9904\\
\hline
 128&  2.5030e-04  &  1.9934 &  2.4450e-02  &  0.9977\\
\hline
 256&  6.2653e-05  &  1.9982 &  1.2252e-02  &  0.9967\\
\hline
 512&  1.5648e-05  &  2.0013 &  6.1377e-03  &  0.9973\\
\hline
\end{tabular}
  \end{minipage}
  \begin{minipage}[t]{0.5\textwidth}
   \centering
        \makeatletter\def\@captype{table}\makeatother
\captionsetup{font={small}}
\caption{\label{tab4}
Finite element errors  for {\bf Example 1} with $\beta_{1}=1$, $\beta_{2}=10^{4}$.}
\small
\begin{tabular}{|c|c|c|c|c|}
\hline
$\frac{1}{h}$          &     $\|u-u_{h}\|_{0,\Omega}$   & order&  $|u-u_{h}|_{1,\Omega}$  &   order    \\
\hline
 32&   3.9444e-03  &         &  9.7007e-02  &         \\
\hline
 64&   9.9671e-04  &  1.9845 &  4.8825e-02  &  0.9904 \\
\hline
 128&  2.5033e-04  &  1.9933 &  2.4450e-02  &  0.9977\\
\hline
 256&  6.2665e-05  &  1.9981 &  1.2253e-02  &  0.9967\\
\hline
 512&  1.5659e-05  &  2.0006 &  6.1379e-03  &  0.9973\\
\hline
\end{tabular}
   \end{minipage}
\end{minipage}

\vspace{1cm}
From Tables \ref{tab11}-\ref{tab12}, we can see that the desired multigrid method converges uniformly with respect to the mesh size and the jump ratio.

\vspace{1cm}
\begin{minipage}{\textwidth}
 \begin{minipage}[t]{0.5\textwidth}
  \centering
     \makeatletter\def\@captype{table}\makeatother
\captionsetup{font={small}}
\caption{\label{tab11}
Numerical performance of Algorithm \ref{vcycle} for {\bf Example 1} with $\beta_1/\beta_2=10^{-4}$.}
\begin{tabular}{|c|c|c|c|c|}
\hline
$h$     & $2^{-6}$ &  $2^{-7}$&  $2^{-8}$  &   $2^{-9}$    \\
\hline
$\#$iter &   8     &  8      &        8  &    8     \\
\hline
\end{tabular}
  \end{minipage}
  \begin{minipage}[t]{0.5\textwidth}
   \centering
        \makeatletter\def\@captype{table}\makeatother
\captionsetup{font={small}}
\caption{\label{tab12}
Numerical performance of Algorithm \ref{vcycle} for {\bf Example 1} with $h=2^{-9}$.}
\begin{tabular}{|c|c|c|c|c|}
\hline
$\beta_1/\beta_2$     & $10^4$ &  $10^2$&  $10^{-2}$  &   $10^{-4}$    \\
\hline
$\#$iter &   8     &  8      &    8 &    8    \\
\hline
\end{tabular}
   \end{minipage}
\end{minipage}
%\begin{figure}[H]\centering
%\includegraphics[width=6cm,height=5cm]{e11.jpg}
%\hspace{0.5cm}
%\includegraphics[width=6cm,height=5cm]{e12.jpg}
%  \caption{Finite element error analysis in log-log scale for Example 1 with $\beta_1/\beta_2=10^4$(left) and $\beta_1/\beta_2=10^2$(right).}
%  \label{e112}
%\end{figure}
%
%\begin{figure}[H]\centering
%\includegraphics[width=6cm,height=5cm]{e13.jpg}
%\hspace{0.5cm}
%\includegraphics[width=6cm,height=5cm]{e14.jpg}
%  \caption{Finite element error analysis in log-log scale for Example 1 with $\beta_1/\beta_2=10^{-2}$(left) and $\beta_1/\beta_2=10^{-4}$(right).}
%  \label{e134}
%\end{figure}

\begin{figure}[H]\centering
\includegraphics[width=6cm,height=5cm]{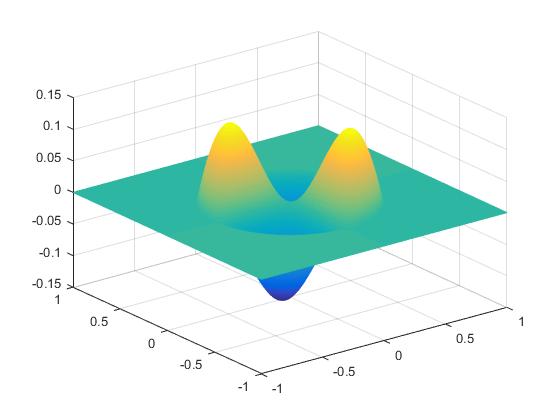}
\hspace{0.5cm}
\includegraphics[width=6cm,height=5cm]{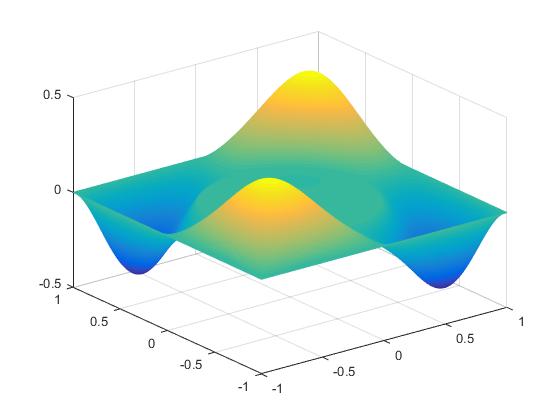}
  \caption{Numerical solutions for {\bf Example 1} with $\beta_1/\beta_2=10^4$(left) and $\beta_1/\beta_2=10^{-4}$(right).}\label{uh112}
\end{figure}

%\begin{figure}[H]\centering
%\includegraphics[width=6cm,height=5cm]{uh11.jpg}
%\hspace{0.5cm}
%\includegraphics[width=6cm,height=5cm]{uh12.jpg}
%  \caption{Numerical solutions for {\bf Example 1} with $\beta_1/\beta_2=10^4$(left) and $\beta_1/\beta_2=10^2$(right).}\label{uh112}
%\end{figure}
%
%\begin{figure}[H]\centering
%\includegraphics[width=6cm,height=5cm]{uh13.jpg}
%\hspace{0.5cm}
%\includegraphics[width=6cm,height=5cm]{uh14.jpg}
%  \caption{Numerical solutions for {\bf Example 1} with $\beta_1/\beta_2=10^{-2}$(left) and $\beta_1/\beta_2=10^{-4}$(right).}\label{uh134}
%\end{figure}
%\begin{table}[H]
%\caption{\label{tab11}
%Numerical performance of Algorithm \ref{vcycle} with $\beta_1/\beta_2=10^{-4}$.}
%\begin{center}
%\begin{tabular}{|c|c|c|c|c|}
%\hline
%$h$     & $2^{-6}$ &  $2^{-7}$&  $2^{-8}$  &   $2^{-9}$    \\
%\hline
%$\#$iter &   8     &  8      &        8  &    8     \\
%\hline
%\end{tabular}
%\end{center}
%\end{table}
%
%\begin{table}[H]
%\caption{\label{tab12}
%Numerical performance of Algorithm \ref{vcycle} with $h=2^{-9}$.}
%\begin{center}
%\begin{tabular}{|c|c|c|c|c|}
%\hline
%$\beta_1/\beta_2$     & $10^4$ &  $10^2$&  $10^{-2}$  &   $10^{-4}$    \\
%\hline
%$\#$iter &   8     &  8      &    8 &    8    \\
%\hline
%\end{tabular}
%\end{center}
%\end{table}

\subsection{Example 2}
The interface is a cardioid curve (see Figure \ref{interface2}),
\begin{equation*}
\phi(x,y)=((x+0.5)^2+y^2-0.5(x+0.5))^2-0.25((x+0.5)^2+y^2).
\end{equation*}
\begin{figure}[H]\centering
\includegraphics[width=5cm,height=4cm]{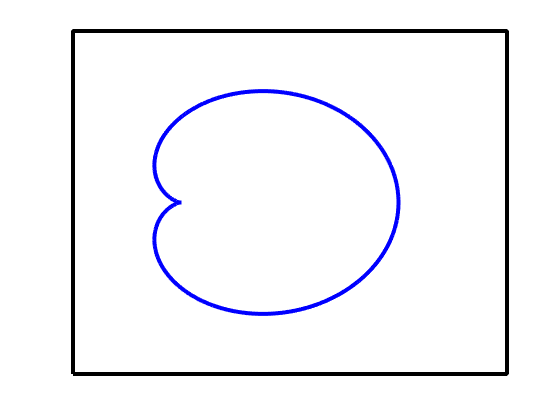}
  \caption{A cardioid interface.}\label{interface2}
\end{figure}
Then the exact solution is chosen as follows
\begin{equation*}
u = \frac{1}{\beta}\sin(\pi x)\sin(\pi y)+5\delta(x,y),
\end{equation*}
where
\begin{eqnarray*}
\delta(x,y)=\left\{
\begin{aligned}
&0,~~(x,y)\in \Omega_{1},\\
&1,~~(x,y)\in \Omega_{2}.
\end{aligned}
\right.
\end{eqnarray*}
We test our method for the second order elliptic interface problem \eqref{elliptic} whose exact solutions are defined as above and whose coefficient jump ratio $\beta_{1}/\beta_{2}=10^3,10^{-3}$, Numerical results are shown in Figure \ref{e2}, illustrating that the convergence rates are optimal in $H^1$-norm and $L^2$-norm. For the non-homogeneous case, Table \ref{tab21}-\ref{tab22} show that our multigrid algorithm is still optimal.
\begin{figure}[H]\centering
\includegraphics[width=6cm,height=5cm]{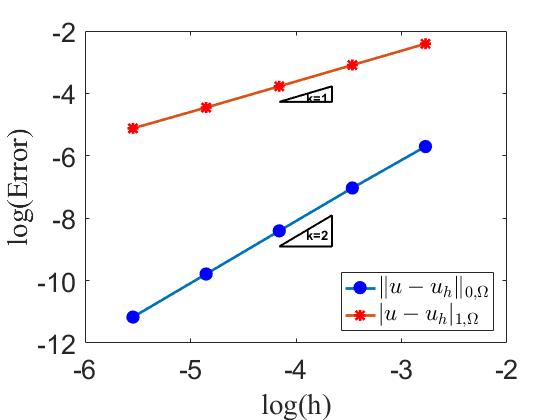}
\hspace{0.5cm}
\includegraphics[width=6cm,height=5cm]{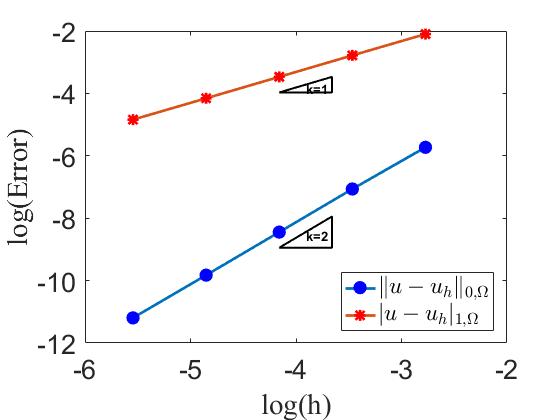}
  \caption{Finite element error analysis in log-log scale for {\bf Example 2} with $\beta_1/\beta_2=10^3$(left) and $\beta_1/\beta_2=10^{-3}$(right).}
  \label{e2}
\end{figure}
\begin{figure}[H]\centering
\includegraphics[width=6cm,height=5cm]{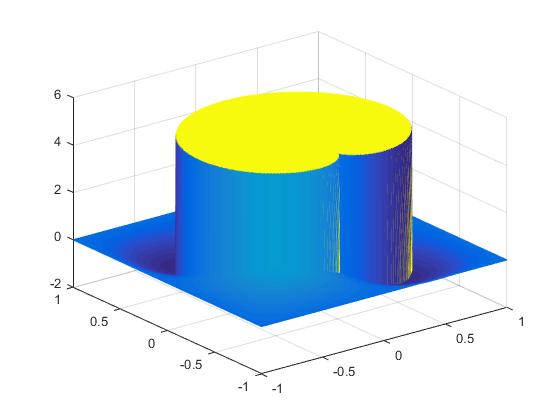}
\hspace{0.5cm}
\includegraphics[width=6cm,height=5cm]{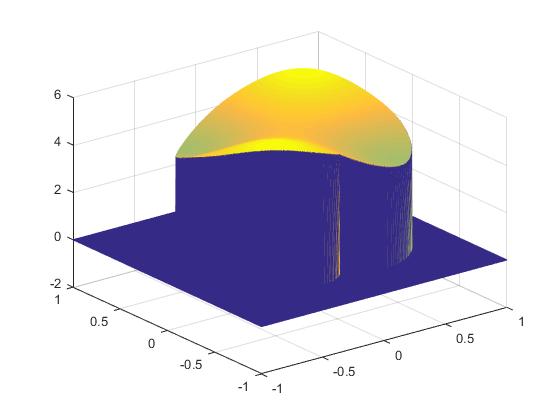}
  \caption{Numerical solutions for {\bf Example 2} with $\beta_1/\beta_2=10^3$(left) and $\beta_1/\beta_2=10^{-3}$(right).}
\end{figure}
\vspace{0.8cm}
\begin{minipage}{\textwidth}
 \begin{minipage}[t]{0.5\textwidth}
  \centering
     \makeatletter\def\@captype{table}\makeatother
\captionsetup{font={small}}
\caption{\label{tab21}
Numerical performance of Algorithm \ref{vcycle} for {\bf Example 2} with $\beta_1/\beta_2=10^{-4}$.}
\begin{tabular}{|c|c|c|c|c|}
\hline
$h$     & $2^{-6}$ &  $2^{-7}$&  $2^{-8}$  &   $2^{-9}$    \\
\hline
$\#$iter &   9     &  9      &        9  &    9     \\
\hline
\end{tabular}
  \end{minipage}
  \begin{minipage}[t]{0.5\textwidth}
   \centering
        \makeatletter\def\@captype{table}\makeatother
\captionsetup{font={small}}
\caption{\label{tab22}
Numerical performance of Algorithm \ref{vcycle} for {\bf Example 2} with $h=2^{-9}$.}
\begin{tabular}{|c|c|c|c|c|}
\hline
$\beta_1/\beta_2$     & $10^4$ &  $10^2$&  $10^{-2}$  &   $10^{-4}$    \\
\hline
$\#$iter &   9    &  9      &    9 &    9    \\
\hline
\end{tabular}
   \end{minipage}
\end{minipage}
%\begin{table}[H]
%\caption{\label{tab21}
%Numerical performance of Algorithm \ref{vcycle} with $\beta_1/\beta_2=10^{-4}$.}
%\begin{center}
%\begin{tabular}{|c|c|c|c|c|}
%\hline
%$h$     & $2^{-6}$ &  $2^{-7}$&  $2^{-8}$  &   $2^{-9}$    \\
%\hline
%$\#$iter &   9     &  9      &        9  &    9     \\
%\hline
%\end{tabular}
%\end{center}
%\end{table}
%
%\begin{table}[H]
%\caption{\label{tab22}
%Numerical performance of Algorithm \ref{vcycle} with $h=2^{-9}$.}
%\begin{center}
%\begin{tabular}{|c|c|c|c|c|}
%\hline
%$\beta_1/\beta_2$     & $10^4$ &  $10^2$&  $10^{-2}$  &   $10^{-4}$    \\
%\hline
%$\#$iter &   9    &  9      &    9 &    9    \\
%\hline
%\end{tabular}
%\end{center}
%\end{table}
\subsection{Example 3}
There are two interfaces in this example, one is a five star curve, the other is a circle ( see Figure \ref{interface3}), i.e,
\begin{equation*}
\phi(x,y)=(\rho_1-0.3 - 0.09\sin(5\theta))(\rho_2^2-0.09),
\end{equation*}
where $\rho_1 = (x+0.5)^2+y^2$, $\rho_2 = (x-0.5)^2+y^2$.
The exact solution is chosen as follows
\begin{equation*}
u = \frac{1}{\beta}\sin(\pi x)\sin(\pi y)+\delta(x,y).
\end{equation*}
We test the local anisotropic FEM for the second order elliptic interface problem \eqref{elliptic} whose exact solutions are defined as above and whose coefficient jump ratio $\beta_{1}/\beta_{2}=10^3,10^{-3}$. Numerical results are shown in Figure \ref{e2}, illustrating that the convergence rates are optimal in $H^1$-norm and $L^2$-norm.
\begin{figure}\centering
\includegraphics[width=5cm,height=4cm]{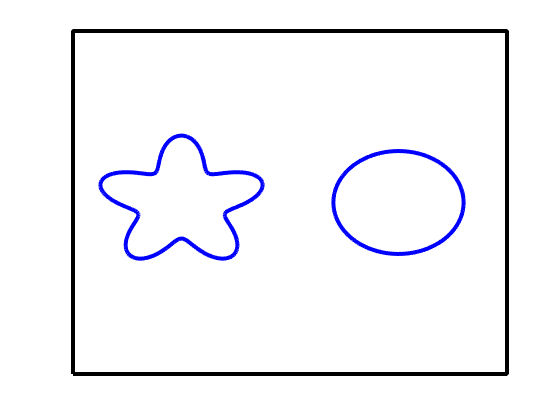}
  \caption{A five star and a circle interfaces.}\label{interface3}
\end{figure}
\begin{figure}\centering
\includegraphics[width=6cm,height=5cm]{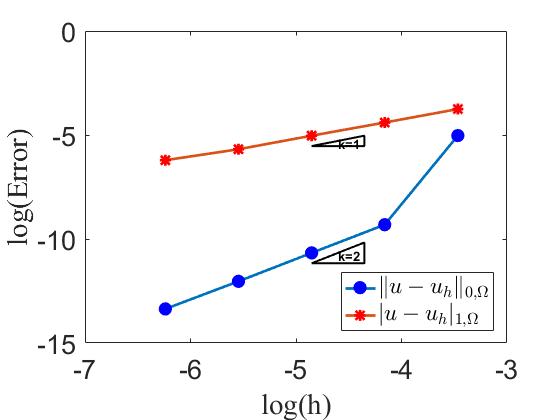}
\hspace{0.5cm}
\includegraphics[width=6cm,height=5cm]{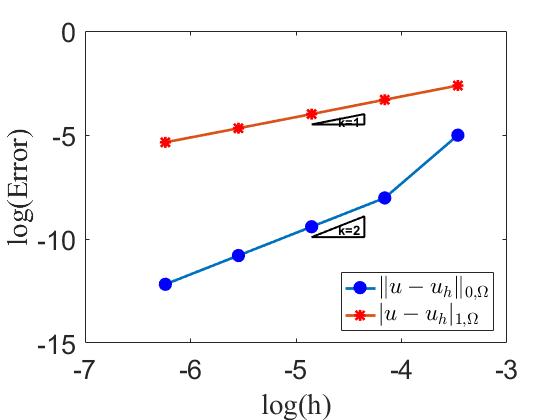}
  \caption{Finite element error analysis in log-log scale for {\bf Example 3} with $\beta_1/\beta_2=10^3$(left) and $\beta_1/\beta_2=10^{-3}$(right).}
  \label{e2}
\end{figure}

\begin{figure}[H]\centering
\includegraphics[width=6cm,height=5cm]{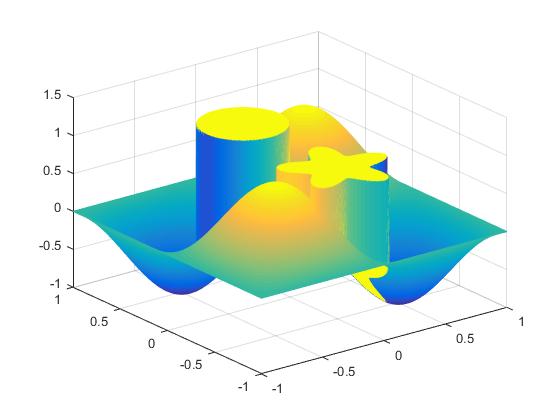}
\hspace{0.5cm}
\includegraphics[width=6cm,height=5cm]{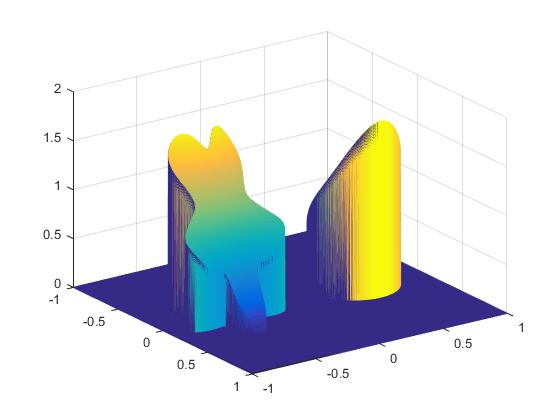}
  \caption{Numerical solutions for {\bf Example 3} with $\beta_1/\beta_2=10^3$(left) and $\beta_1/\beta_2=10^{-3}$(right).}
\end{figure}

Still Table \ref{tab21}-\ref{tab22} show that our multigrid algorithm has an optimal convergence rate independent of the mesh size and jump ratio.

\vspace{0.8cm}
\begin{minipage}{\textwidth}
 \begin{minipage}[t]{0.5\textwidth}
  \centering
     \makeatletter\def\@captype{table}\makeatother
\captionsetup{font={small}}
\caption{\label{tab31}
Numerical performance of Algorithm \ref{vcycle} for {\bf Example 3} with $\beta_1/\beta_2=10^{-4}$.}
\begin{tabular}{|c|c|c|c|c|}
\hline
$h$     & $2^{-6}$ &  $2^{-7}$&  $2^{-8}$  &   $2^{-9}$    \\
\hline
$\#$iter &   8     &  8      &        8  &    8    \\
\hline
\end{tabular}
  \end{minipage}
  \begin{minipage}[t]{0.5\textwidth}
   \centering
        \makeatletter\def\@captype{table}\makeatother
\captionsetup{font={small}}
\caption{\label{tab32}
Numerical performance of Algorithm \ref{vcycle} for {\bf Example 3} with $h=2^{-9}$.}
\begin{tabular}{|c|c|c|c|c|}
\hline
$\beta_1/\beta_2$     & $10^4$ &  $10^2$&  $10^{-2}$  &   $10^{-4}$    \\
\hline
$\#$iter &   8     &  8      &    8 &   8    \\
\hline
\end{tabular}
   \end{minipage}
\end{minipage}

%\begin{table}[H]
%\caption{\label{tab31}
%Numerical performance of Algorithm \ref{vcycle} with $\beta_1/\beta_2=10^{-4}$.}
%\begin{center}
%\begin{tabular}{|c|c|c|c|c|}
%\hline
%$h$     & $2^{-6}$ &  $2^{-7}$&  $2^{-8}$  &   $2^{-9}$    \\
%\hline
%$\#$iter &   8     &  8      &        8  &    8    \\
%\hline
%\end{tabular}
%\end{center}
%\end{table}
%
%\begin{table}[H]
%\caption{\label{tab32}
%Numerical performance of Algorithm \ref{vcycle} with $h=2^{-9}$.}
%\begin{center}
%\begin{tabular}{|c|c|c|c|c|}
%\hline
%$\beta_1/\beta_2$     & $10^4$ &  $10^2$&  $10^{-2}$  &   $10^{-4}$    \\
%\hline
%$\#$iter &   8     &  8      &    8 &   8    \\
%\hline
%\end{tabular}
%\end{center}
%\end{table}
%\section{conclusions}
%For interface problems, we have proposed a simple way to generate an interface-fitted mesh quickly. Although the generated mesh contains some anisotropic elements around the interface, we prove that the corresponding finite element space still have optimal convergence rates in both $H^1$-norm and $L^2$-norm. Furthermore, we have desired a multigrid iterator for the linear system. The convergence rate of the multigrid method is uniform with respect to mesh size and coefficient jump ratio.
\bibliographystyle{plain}
\bibliography{refer}
%\bibliography{refer}
\end{document}